\documentclass[12pt,reqno]{amsart}

\usepackage{ amsmath }
\usepackage{amsmath , amssymb, latexsym }
\usepackage{graphics}
\usepackage{graphicx}
\usepackage{verbatim}
\usepackage{color}
\usepackage{enumerate}
\usepackage{hyperref} 

\usepackage[mathscr]{eucal}

\newtheorem{thm}{Theorem}[section]
\newtheorem{lemma}[thm]{Lemma}
\newtheorem{cor}[thm]{Corollary}
\newtheorem{prop}[thm]{Proposition}
\newtheorem{obs}[thm]{Observation}

\newtheorem{conj}[thm]{Conjecture}
\newtheorem{qu}{Question}
\newtheorem{prob}{Problem}
\newtheorem*{Chernoff}{Chernoff's inequality}
\newtheorem*{Bey}{Bey's inequality}

\newtheorem*{RSWvac}{Vacant RSW Theorem}
\theoremstyle{definition}\newtheorem*{alg}{The Water Algorithm, $\A_W$}
\theoremstyle{definition}
\theoremstyle{definition}\newtheorem{defn}{Definition}

\newcommand{\ds}{\displaystyle}

\makeatletter
\def\Ddots{\mathinner{\mkern1mu\raise\p@
\vbox{\kern7\p@\hbox{.}}\mkern2mu
\raise4\p@\hbox{.}\mkern2mu\raise7\p@\hbox{.}\mkern1mu}}
\makeatother

\def\eps{\Varepsilon}
\def\->{\rightarrow}

\newcommand{\1}{\mathbf{1}}
\newcommand{\E}{\mathbb{E}}

\newcommand{\N}{\mathbb{N}}
\newcommand{\ZZ}{\mathbb{Z}}

\def\A{\mathcal{A}}

\def\C{\mathcal{C}}

\def\F{\mathcal{F}}

\def\HH{\mathcal{H}}

\def\N{\mathcal{N}}

\def\P{\mathbb{P}}
\def\PP{\mathscr{P}}

\def\PB{\mathbf{P}}
\def\ExB{\mathbf{E}}
\def\VarB{\mathbf{Var}}

\def\AA{\mathbb{A}}
\def\Ex{\mathbb{E}}

\def\N{\mathbb{N}}
\def\Pr{\mathbb{P}}

\def\RR{\mathbb{R}}
\def\TT{\mathbb{T}}
\def\ZZ{\mathbb{Z}}

\def\le{\leqslant}
\def\ge{\geqslant}

\def\eps{\varepsilon}
\def\<{\langle}
\def\>{\rangle}

\def\Bin{\textup{Bin}}

\def\be{\begin{equation}}
\def\ee{\end{equation}}
\def\bea{\begin{equation*}}
\def\eea{\end{equation*}}

\def\Var{{\rm Var}}
\def\Cov{{\rm Cov}}

\def\Inf{{\rm Inf}}

\title{Noise Sensitivity in Continuum Percolation}

\author{Daniel Ahlberg}
\address{Department of Mathematical Sciences, University of Gothenburg, and Department of Mathematical Sciences, Chalmers University of Technology}\email{md1ahlda@chalmers.se}

\author{Erik Broman}
\address{Mathematical Sciences Chalmers University of Technology and Mathematical Sciences G\"oteborg University SE-41296 Gothenburg, Sweden} \email{broman@chalmers.se}

\author{Simon Griffiths}
\address{IMPA, Estrada Dona Castorina 110, Jardim Bot\^anico, Rio de Janeiro, RJ, Brasil} \email{sgriff@impa.br}

\author{Robert Morris}
\address{IMPA, Estrada Dona Castorina 110, Jardim Bot\^anico, Rio de Janeiro, RJ, Brasil} \email{rob@impa.br}

\thanks{Research supported in part by: (DA) The Royal Swedish Academy of Sciences; (EB) The G\"oran Gustafsson Foundation for Research in Natural Sciences and Medicine; (SG) CNPq bolsa PDJ; (RM) CNPq bolsa de Produtividade em Pesquisa.}

\begin{document}

\begin{abstract}
We prove that the Poisson Boolean model, also known as the Gilbert disc model, is noise sensitive at criticality. This is the first such result for a Continuum Percolation model, and the first which involves a percolation model with critical probability $p_c \ne 1/2$. Our proof uses a version of the Benjamini-Kalai-Schramm Theorem for biased product measures. A quantitative version of this result was recently proved by Keller and Kindler. We give a simple deduction of the non-quantitative result from the unbiased version. We also develop a quite general method of approximating Continuum Percolation models by discrete models with $p_c$ bounded away from zero; this method is based on an extremal result on non-uniform hypergraphs.
\end{abstract}

\maketitle

\section{Introduction}

The concept of noise sensitivity of a sequence of Boolean functions was introduced in 1999 by Benjamini, Kalai and Schramm~\cite{BKS}, and has since developed into one of the most exciting areas in Probability Theory, linking Percolation with Discrete Fourier Analysis and Combinatorics. 
So far, most attention has been focused on percolation crossings in two dimensions~\cite{BKS,GPS,SS},
either for bond percolation on the square lattice $\ZZ^2$, or for site percolation on the triangular lattice $\TT$. In this paper we study the corresponding questions in the setting of Continuum Percolation; in particular, we shall prove that the Poisson Boolean model, also known as the Gilbert disc model, is noise sensitive at criticality.

Roughly speaking, a sequence of Boolean functions $f_n \colon \{0,1\}^n \to \{0,1\}$ is said to be \emph{noise sensitive} if a slight perturbation of the state $\omega$ asymptotically causes all information about $f_n(\omega)$ to be lost. More precisely, let $\eps > 0$ and suppose that $\omega \in \{0,1\}^n$ is chosen uniformly at random. Define $\omega^\eps \in \{0,1\}^n$ to be the (random) state obtained by re-sampling each coordinate (independently and uniformly) with probability $\eps$, and note that $\omega^\eps$ is also a uniform element of $\{0,1\}^n$. Then the sequence $(f_n)_{n \ge 1}$ is said to be \emph{noise sensitive} (NS) if, for every $\eps > 0,$
\begin{equation}\label{def:NS}
\lim_{n \to \infty} \Ex\big[ f_{n}(\omega) f_{n}( \omega^{\eps}) \big] - \Ex\big[ f_{n}(\omega) \big]^2 \,=\, 0.
\end{equation}
For example, the Majority function ($f_n(\omega) = 1$ iff $\sum \omega_j > n/2$) and the Dictator function ($f_n(\omega) = 1$ iff $\omega_1 = 1$) are not noise sensitive, but the Parity function ($f_n(\omega) = 1$ iff $\sum \omega_j$ is even) is noise sensitive. One can easily see, using the Fourier representation of Section \ref{BKSsec}, that if~\eqref{def:NS} holds for some $\eps\in(0,1)$, then it holds for every $\eps \in (0,1)$.

Noise sensitivity was first defined by Benjamini, Kalai and Schramm~\cite{BKS}, who were partly motivated by the problem of exceptional times in dynamical percolation (see~\cite{Steif}). In this model, which was introduced independently by Benjamini (unpublished) and by H\"aggstr\"om, Peres and Steif~\cite{HPS1}, each bond in $\ZZ^2$ (or site in $\TT$) has a Poisson clock, and updates its state (again according to the uniform distribution) every time the clock rings. At any given time, the probability that there is an infinite component of open edges is zero (see~\cite{BR} or~\cite{Grim}, for example). However, there might still exist certain exceptional times at which such a component appears. Building on the work of~\cite{BKS}, Schramm and Steif~\cite{SS} were able to prove that, for the triangular lattice $\TT$, such exceptional times do exist, and moreover the Hausdorff dimension of the set of such times lies in $[1/6,31/36]$. Even stronger results were obtained by Garban, Pete and Schramm~\cite{GPS}, who were able to prove, via an extremely precise result on the Fourier spectrum of the `percolation crossing event', that the dimension of the exceptional set for $\TT$ is $31/36$, and that exceptional times also exist for bond percolation on $\ZZ^2$.

Following~\cite{BKS}, we shall study Boolean functions which encode `crossings' in percolation models. For example, consider bond percolation on $\ZZ^2$ at criticality (i.e., with $p = p_c = 1/2$), and let $f_N$ encode the event that there is a horizontal crossing of $R_N$, the $N \times N$ square centred at the origin, using only the open edges of the configuration. In other words, let $f_N \colon \{0,1\}^E \to \{0,1\}$, where $E$ is the set of edges of $\ZZ^2$ with an endpoint in $R_N$, be defined by $f_N(\omega) = 1$ if and only if there is such a crossing using only edges $e \in E$ with $\omega_e = 1$. In~\cite{BKS}, the authors proved that the sequence $(f_N)_{N \in \N}$ is noise sensitive.

Continuum Percolation describes the following family of random geometric graphs: first choose a countable random subset of $\RR^d$ according to some distribution, and then join two points with an edge in a deterministic way, based on their relative position. Two especially well-studied examples are Voronoi percolation (see~\cite{BS1,BRvor}), and the Poisson Boolean model, which was introduced by Gilbert in 1961~\cite{Gilbert}, and studied further in~\cite{Alex,BBW,BS2,MS,Roy}. In the latter model, a set $\eta$ of points in the plane $\RR^2$ are chosen according to a Poisson point process with intensity $\lambda$, and for each point $x \in \eta$, a closed disc of radius 1 is placed with its centre on $x$; let $D(\eta)$ denote the union of these discs. The model is said to \emph{percolate} if there exists an unbounded connected component in $D(\eta)$. It is well known that there exists a critical intensity $0 < \lambda_c < \infty$ such that if $\lambda < \lambda_c$ then the model a.s. does not percolate, while if $\lambda > \lambda_c$ it a.s. percolates. See the books~\cite{MR} and~\cite{BR} for a detailed introduction to Continuum Percolation.

We shall be interested in the problem of noise sensitivity of the Poisson Boolean model at criticality, that is, with $\lambda = \lambda_c$.  Let $f^G_N$ be the function, defined on countable, discrete subsets $\eta$ of the plane $\RR^2$, which encodes whether or not there is a horizontal crossing of $R_N$ using only points of $D(\eta) \cap R_N$. That is,
$$f^G_N(\eta) = 1 \quad \Leftrightarrow \quad H\big( \eta, R_N, \bullet \big) \textup{ occurs,}$$
where $H( \eta, R_N, \bullet)$ denotes the event that such a crossing exists in the `occupied space' $D(\eta)$.

Since $f^G_N$ is defined on a continuous state space, we shall need to modify the definition of noise sensitivity. Let $\eps > 0$ and~$\lambda > 0$, and let $\eta \subset \RR^2$ be chosen according to a Poisson point process of intensity~$\lambda$. We shall denote the measure associated to this Poisson process by $\PB_\lambda$, expectation with respect to this measure by $\ExB_\lambda$, and variance $\VarB_\lambda$. We define $\eta^\eps$ to be the random subset of $\RR^2$ obtained by deleting each element of $\eta$ independently with probability $\eps$, and then adding a independent Poisson point process of intensity $\eps \lambda$. It is clear that $\eta^\eps$ has the same distribution as $\eta$, and we will for that reason allow a minor abuse of notation letting $\PB_\lambda$ denote the measure by which the pair $(\eta,\eta^\eps)$ is chosen.

\begin{defn}[Noise sensitivity for Continuum Percolation] \label{def:NScts}
We say that the Poisson Boolean model is \emph{noise sensitive} at $\lambda$ if the sequence of functions $(f^G_N)_{N \ge 1}$ satisfies
\[
\lim_{N \to \infty} \ExB_\lambda\big[ f^G_N(\eta) f^G_{N}( \eta^{\eps}) \big] - \ExB_\lambda\big[ f^G_N(\eta) \big]^2 \, =\, 0 \qquad\qquad \textup{for every } \eps > 0.
\]
We shall say that the model is \emph{noise sensitive at criticality} if it is noise sensitive at $\lambda_c$.
\end{defn}

We remark that the Poisson Boolean model is trivially noise sensitive at every $\lambda \ne \lambda_c$. The reason is simply that when $\lambda > \lambda_c$ (or $\lambda<\lambda_c$), then $\lim_N f_N^G=1$ a.s.
(or $\lim_N f_N^G=0$ a.s.), as is well known.

The following theorem is our main result. It is the analogue for the Poisson Boolean model of the result from~\cite{BKS} mentioned above concerning bond percolation on $\ZZ^2$.

\begin{thm}\label{noise}
The Poisson Boolean model is noise sensitive at criticality.
\end{thm}

The proof of Theorem~\ref{noise} is based on two very general theorems, neither of which uses any properties of the specific model which we are studying. The first is a generalization of one of the main theorems of Benjamini, Kalai and Schramm~\cite{BKS}, a result referred to as the BKS Theorem, to biased product measures. It gives a sufficient condition (based on the concept of influence) for an arbitrary sequence of functions to be noise sensitive at density $p$ (see Theorem~\ref{BKSp}). A quantitative version of the BKS Theorem for biased product measures was recently proved by Keller and Kindler~\cite{KK}. Their result is therefore a strengthening of the qualitative result of~\cite{BKS}. We shall give a short deduction of the BKS Theorem for general $p \in (0,1)$ from the uniform case.

The second main tool is an extremal result on arbitrary non-uniform hypergraphs (i.e., arbitrary events on $\{0,1\}^n$), which allows us to bound the variance that arises when two stages of randomness are used to choose a random subset.  We shall use this bound (see Theorem~\ref{varB}) to prove noise sensitivity for the Poisson Boolean model via a corresponding result for a particular discrete percolation model (see Theorem~\ref{BpNS}). These tools are quite general, and we expect both to have other applications; we shall therefore state them here, and in some detail, for easy reference.

In order to state the BKS Theorem for product measure, we first need to define noise sensitivity in this setting. Let $\P_p$ denote product measure with density $p \in (0,1)$ on $\{0,1\}^n$, i.e., $\P_p(\omega_i = 1) = p$ independently for every $i \in [n]:=\{1,2,\ldots,n\}$. (We will let $\Ex_p$ denote expectation with respect to this measure.) When $p = 1/2$ this corresponds to picking an element of $\{0,1\}^n$ uniformly at random, and so we refer to it as the uniform case. Define $\omega^{\eps}$ as above, by re-randomizing  each bit independently with probability $\eps$.

\begin{defn}[Noise sensitivity at density $p$]\label{def:NSp}
A sequence of functions $f_{n} \colon \{0,1\}^n\rightarrow [0,1]$ is said to be \emph{noise sensitive at density $p$} (NS$_p$) if, for every $\eps > 0$,
\begin{equation} \label{eqn:NS_p}
\lim_{n\rightarrow \infty}  \E_p\big[ f_{n}(\omega) f_{n}(\omega^{\eps}) \big] - \E_p\big[ f_{n}(\omega) \big]^2 \,=\, 0.
\end{equation}
When $p = 1/2$, this is equivalent to~\eqref{def:NS}, the definition of noise sensitivity from~\cite{BKS}.
\end{defn}

The \emph{influence at density $p$}, denoted $\Inf_{p,i}(f)$, of a coordinate $i \in [n]$ in a function $f \colon \{0,1\}^n \to [0,1]$, is defined by
$$\Inf_{p,i}(f) \, := \, \E_p\big[ \big| f(\omega) - f(\sigma_i \omega) \big| \big],$$
where $\sigma_i$ is the function that flips the value of $\omega$ at position $i$. We denote the sum of the squares of the influences of $f$ by
$$I\!I_p(f) \, := \, \sum_{i=1}^{n}\Inf_{p,i}(f)^2.$$
The following theorem was first proved by Benjamini, Kalai and Schramm~\cite{BKS} in the case $p = 1/2$, but also remarked to hold for general $p$. A quantitative version was obtained by Keller and Kindler~\cite{KK}. The result was stated in~\cite{BKS} for functions into $\{0,1\}$, but the proof works also for functions into $[0,1]$, as was also observed in \cite[page 3]{KK}. We shall give a simple deduction of this theorem from the uniform case.

\begin{thm}[BKS Theorem for product measure]\label{BKSp}
Let $(f_n)_{n \ge 1}$ be a sequence of functions $f_n \colon \{0,1\}^n \to [0,1]$. For every $p\in(0,1)$,
$$\lim_{n\to\infty}I\!I_p(f_n)=0 \qquad \Rightarrow \qquad (f_n)_{n \ge 1} \textup{ is NS$_p$}.$$
\end{thm}

We remark that the approach we use to prove Theorem~\ref{BKSp} is quite general, and may be used to extend various other results from uniform to biased product measures, see Section~\ref{hfsec}. Before introducing our second main tool, Theorem~\ref{varB}, let us give some more context, by describing our general approach to the proof of Theorem~\ref{noise}.

We shall choose our Poisson configuration $\eta \subset \RR^2$ in two steps; in other words, we view the Poisson Boolean model as a `weighted average' (according to a certain probability distribution) of a family of discrete percolation models. To be precise, for each countable set $B \subset \RR^2$ and $p \in (0,1)$ we consider the following simple model; it is nothing more than site percolation on the graph (with vertex set $B$) defined by the Poisson Boolean model. 

\begin{defn}[The percolation model $\PP^{B}_p$]
For each countable $B \subset \RR^2$ and $p \in (0,1)$, a configuration $\eta\subset B$ in the percolation model $\PP^{B}_{p}$ is obtained by including each point of $B$ independently with probability $p$.
\end{defn}

A \emph{$p$-subset} of a (countable) set $S$ is a random subset chosen by including each element independently with probability $p$. We shall write $\Pr^S_p$ for the corresponding probability distribution; or just $\Pr_p$ when $S = [n]$. We will be interested in the $\PP^B_p$ model for fixed (small) $p>0$ in the case where $B\subset\RR^2$ is chosen according to $\PB_{\lambda_c/p}$, the measure of a Poisson point process with intensity $\lambda_c/p$. Given $B$ chosen in that way, let~$\eta$ denote a $p$-subset of $B$. That is, choose $\eta\subset B$ according to the conditional measure $\Pr^B_p$. We emphasize that choosing~$\eta$ in this two-step procedure is equivalent to choosing it according to a Poisson point process of intensity~$\lambda_c$. Consequently, $D(\eta)$ corresponds to a configuration of the occupied space in the Poisson Boolean model at criticality.
  
For each countable set $B \subset \RR^2$, define the function $f_{N}^{B} \colon \{0,1\}^{B}\to \{0,1\}$ by setting $f^B_N(\eta) = 1$ if and only if there is a horizontal crossing of $R_N$ in $D(\eta)$. (Note that although $B$ may be infinite, the function $f_N^B$ only depends on a finite number of elements of $B$, a.s. in~$\PB_{\lambda_c/p}$.) We say that the model $\PP^B_p$ is noise sensitive at density $p$ (NS$_p$) if the sequence $(f^B_N)_{N \ge 1}$ is NS$_p$. Our proof of Theorem~\ref{noise} proceeds via the following result.

\begin{thm}\label{BpNS}
The model $\PP^B_p$ is noise sensitive at density $p$ for $\PB_{\lambda_c/p}$-almost every $B$, for each sufficiently small $p > 0$.
\end{thm}

Thus, the proof of Theorem~\ref{noise} divides naturally into two parts.  In the first we adapt the methods of Benjamini, Kalai and Schramm~\cite{BKS} to prove noise sensitivity of the discrete models $\PP^{B}_{p}$; in the second we use our bound on the variance (Theorem~\ref{varB}, below) to prove that this noise sensitivity transfers to the continuous Poisson Boolean model. Interestingly, Theorem~\ref{varB} will also be a key tool in the proof of Theorem~\ref{BpNS}.

To apply the adapted methods of~\cite{BKS}, we require a bound on the fluctuations (in $B$ chosen according to $\PB_{\lambda_c/p}$) of the probability of the crossing event in the model $\PP^B_p$; we shall prove such a bound in a \emph{much} more general context. Indeed, our next theorem holds for \emph{arbitrary} hypergraphs (events), and thus we shall not use any properties of the specific percolation problem under consideration. By working at this level of generality, we are able to deduce, with no extra effort, similar bounds for crossings of rectangles, and for crossings in other percolation models.

A hypergraph $\HH$ is simply a collection of subsets of $[n]$; or, equivalently, it is a subset of $\{0,1\}^n$. We call these sets `edges', and remark that if every edge has exactly two elements then $\HH$ is a graph. Given a hypergraph $\HH$, for each set $B \subset [n]$ and $p \in (0,1)$ we define
$$r_{\HH}(B,p) \; := \; \Pr^B_p\big( A \in \HH \big),$$
where $A$ is a $p$-subset of $B$. We remark that our use of the letter $B$ here is supposed to be suggestive; in our applications it will correspond to a discrete approximation of the subset of $\RR^2$ considered above, which was chosen according to $\PB_{\lambda_c/p}$. Indeed, given a rectangle $R \subset \RR^2$, we shall discretize by partitioning it into $n$ small squares. The set $B$ will be a $q$-subset of $[n]$, where $q = q(n)$ is chosen so that (the set of centres of squares corresponding to) $B$ has intensity $\lambda_c/p$ in $R$. Our hypergraph $\HH$ will encode crossings of $R$ in $D(\eta)$, where $\eta$ is a $p$-subset of these squares. Observe that in this case, if $n$ is large, then $r_\HH(B,p)$ is very close to the probability of a crossing of $R$ in the model $\PP^B_{p}$.

\begin{thm}\label{varB}
Let $0 < p \le 1/2$, $0 < q < 1$ and $n \in \N$ satisfy $n \ge 200(pqn)^3$, $n\ge 8p(qn)^2$ and $pqn \ge 32 \log\frac{1}{p}$. Let $\HH$ be a hypergraph on vertex set $[n]$, and let $B$ be a $q$-subset of $[n]$. Then
$$\Var_q\big( r_{\HH}(B,p) \big) \, = \, O\bigg( p \left( \log \frac{1}{p} \right)^2 \bigg),$$
where the constant implicit on the right-hand side is independent of $\HH$.
\end{thm}

We emphasize the crucial point, which is that our bound on $\Var_{q} \big( r_{\HH}(B,p) \big)$ goes to zero as $p \to 0$ \emph{uniformly} in $\HH$. Here, and throughout, $f(x) = O(g(x))$ denotes the existence of an absolute constant $C > 0$, independent of all other variables, such that $|f(x)| \le C|g(x)|$ for every $x$ in the domain of $f$ and $g$. We remark that an alternative proof of Theorem~\ref{varB} was recently obtained by one of us~\cite{Daniel}, using completely different methods.

As noted above, we shall use Theorem~\ref{varB} in order to prove that the
sequence $(f_N^B)_{N\ge 1}$ is noise sensitive for $\PB_{\lambda_c/p}$-almost every $B$, as well as to deduce Theorem~\ref{noise} from Theorem~\ref{BpNS}. Indeed, we shall use Theorem~\ref{varB} together with the `deterministic algorithm' method (see Sections~\ref{algsubsec} and~\ref{Algsec}) to obtain bounds on the influences of variables; Theorem~\ref{BpNS} then follows from the BKS Theorem for product measure. 

We study in this paper how methods developed to study noise sensitivity for discrete percolation models can be adapted to a continuum setting. We have chosen to follow the approach of Benjamini, Kalai, and Schramm~\cite{BKS}.
More recently, \emph{quantitative} noise sensitivity has been introduced. Here one aims at determining
the rate at which $\eps=\eps(n)$ is allowed to decay while the limit in~\eqref{def:NS} persists. Results of this kind were obtained by Schramm and Steif~\cite{SS} and Garban, Pete, and Schramm~\cite{GPS} in the context of percolation crossings on the square lattice $\ZZ^2$ and triangular lattice $\TT^2$. As a corollary of our main result (Theorem~\ref{noise}), via the approach developed in~\cite{SS}, it is possible to obtain similar quantitative results for the Poisson Boolean model.

Define the \emph{noise sensitivity exponent} (for the Poisson Boolean model) as the supremum over the set of $\alpha\ge0$ for which the limit in Definition~\ref{def:NScts} holds with $\eps=\eps(N)=N^{-\alpha}$.

\begin{cor}\label{cor:QNS}
The noise sensitivity exponent for the Poisson Boolean model at criticality is strictly positive. That is, there exists $\alpha>0$ such that, for $\eps(N)=N^{-\alpha}$,
$$
\lim_{N\to\infty}\ExB_{\lambda_c}\big[f_N^G(\eta)f_N^G(\eta^{\eps(N)})\big]-\ExB_{\lambda_c}\big[f_N^G(\eta)\big]^2=0.
$$
\end{cor}

We shall, in Section~\ref{sec:QNS}, only outline the proof of Corollary~\ref{cor:QNS}. The rest of the paper is organized as follows. In Section~\ref{sketchsec} we give a full overview of the proof, and state several other results which may be of independent interest.  In Section~\ref{nontrivsec} we recall some facts about the Poisson Boolean model, and in Sections~\ref{BKSsec} and~\ref{Algsec} we prove Theorem~\ref{BKSp} and extend the deterministic algorithm method of~\cite{BKS} to general $p \in (0,1)$. In Section~\ref{hypersec} we prove Theorem~\ref{varB}, and deduce some simple consequences, in Section~\ref{T1sec} we prove Theorem~\ref{BpNS} and deduce Theorem~\ref{noise}, and in Section~\ref{sec:QNS} we sketch the proof of Corollary~\ref{cor:QNS}. Finally, in Section~\ref{opensec} we state some open questions.

Throughout the article we treat elements of $\{0,1\}^n$ as subsets of $[n]$, and vice versa, without comment, by identifying sets with their indicator functions. Thus $\Pr_p$ denotes the biased product measure on $\{0,1\}^n$, and also the probability distribution associated with choosing a $p$-subset of $[n]$. As remarked above, we shall (suggestively) use the letter $B$ to denote both a $q$-subset of $[n]$, and a subset of $\RR^2$ chosen according to $\PB_{\lambda_c/p}$ (i.e., according to a Poisson point process with intensity $\lambda_c / p$), and trust that this will not confuse the reader. The letter $\eta$ will always denote a random subset of the plane, chosen according to a Poisson process of intensity $\lambda_c$; or, equivalently, chosen as a $p$-subset of the set $B \subset \RR^2$.

\section{Further results, and an overview of the proof}\label{sketchsec}

In this section we introduce a number of auxiliary methods and results that we shall use in the proof of Theorem~\ref{noise}, and which may also be of independent interest.  In particular, we introduce a new way of deducing results for biased product measures from results in the uniform case.  We shall use this method in Sections~\ref{BKSsec} and~\ref{Algsec} to generalize the BKS Theorem and the deterministic algorithm method of Benjamini, Kalai and Schramm~\cite{BKS}.

Let us begin by examining the link between the Poisson Boolean model and the model $\PP^B_p$. It is this link that will enable us to deduce Theorem~\ref{noise} from Theorems~\ref{BpNS} and~\ref{varB}. First, to illustrate the sense in which the Boolean model may be viewed as an `average' of the discrete percolation models $\PP^B_p$, let us fix $p\in (0,1)$, and observe that
\begin{equation}\label{average}
\ExB_{\lambda_c/p}\big[ \E_p^B \big( f_{N}^{B}(\eta) \big) \big] \,  = \, \ExB_{\lambda_c/p}\big[ \P^B_p\big( H( \eta, R_N, \bullet) \big) \big] \, = \, \PB_{\lambda_c} \big(H ( \eta, R_N, \bullet)\big) \, = \, \ExB_{\lambda_c} \big[ f^G_N(\eta) \big],
\end{equation}
where the second equality follows since if $B\subset \RR^2$ is chosen according to a Poisson point process of intensity $\lambda_c /p$, then $\eta$ is distributed as a Poisson point process of intensity $\lambda_c$. Furthermore, it is not difficult to show that for $\eps\in(0,1-p)$ and
$\eps'=\eps/(1-p)$
\begin{equation}\label{eq:NSkey}
\begin{aligned}
\ExB_{\lambda_c}\big[f_N^G(\eta)f_N^G(\eta^\eps)\big]-\ExB_{\lambda_c}\big[f_N^G(\eta)\big]^2\;&=\;\ExB_{\lambda_c/p}  \Big[ \Ex^B_p\big[ f^B_N (\eta) f^B_N (\eta^{\eps'}) \big] - \Ex^B_p\big[ f^B_N (\eta) \big]^2 \Big] \\
&\quad\;+\, \VarB_{\lambda_c/p} \Big( \Ex^B_p\big[ f^B_N (\eta) \big] \Big),
\end{aligned}
\end{equation}
where, as the notation suggests, on the left-hand side $(\eta,\eta^\eps)$ is specified as in Definition~\ref{def:NScts}, and on the right-hand side $(\eta,\eta^{\eps'})$ is chosen as in Definition~\ref{def:NSp}, as subsets of $B$ (see Section~\ref{T1sec}). Thus, proving that the Poisson Boolean model is noise sensitive at criticality reduces to proving that (for some $p$)
\begin{equation}\label{21triv}
\ExB_{\lambda_c/p}  \Big[ \Ex^B_p\big[ f^B_N (\eta) f^B_N (\eta^\eps) \big] - \Ex^B_p\big[ f^B_N (\eta) \big]^2 \Big] \,+\, \VarB_{\lambda_c/p} \Big( \Ex^B_p\big[ f^B_N (\eta) \big] \Big) \, \to \, 0
\end{equation}
as $N \to \infty$. Theorem~\ref{BpNS} says exactly that, for fixed $p$, the first term is $o(1)$ as $N \to \infty$, while the following proposition shows that the second term can be made arbitrarily small by choosing $p$ appropriately. 

\begin{prop}\label{varBp}
$$\lim_{p\to 0}\, \limsup_{N \to \infty} \, \VarB_{\lambda_c/p} \Big( \P^B_p\big(H ( \eta, R_{N}, \bullet)\big) \Big)\, =\, 0.$$
\end{prop}

We shall deduce the proposition from Theorem~\ref{varB} via a straightforward discretization argument. Since the expression in~\eqref{21triv} in fact is independent of $p$ (as seen in~\eqref{eq:NSkey}), it follows that~\eqref{21triv} holds (see Section~\ref{T1sec} for more details).

Having described how we shall deduce Theorem~\ref{noise} from Theorems~\ref{BpNS} and~\ref{varB}, we continue this section with a few comments on their proofs, as well as a presentation of some results and methods we shall use for that purpose.

The proof of Theorem~\ref{varB}, given in Section~\ref{hypersec}, does not rely on the remainder of the paper, but follows instead from an inequality due to Bey~\cite{Bey} (see Section~\ref{hypersubsec} for a brief overview). The proof of noise sensitivity in the $\PP^B_p$ model (Theorem~\ref{BpNS}) is given in Section~\ref{T1sec}, and is based on the approach developed by Benjamini, Kalai, and Schramm~\cite{BKS}. Our main challenge will be to extend their method to the non-uniform case; in particular, we shall need to prove the BKS Theorem for biased product measure (Theorem~\ref{BKSp}), and to generalize a result linking the revealment of algorithms to the influence of bits (see Section~\ref{algsubsec}). In order to do so, we introduce a new method for reducing problems for biased product measure to the uniform case. This method is introduced in Section~\ref{hfsec}; in Section~\ref{BKSsec} we use it to prove the BKS Theorem for product measure, and in Section~\ref{Algsec} to complete the extension of the deterministic algorithm approach.



We begin by stating the key property of the Poisson Boolean model that we shall need.

\subsection{Non-triviality of crossing probabilities}

If the probability (in $\PB_{\lambda_c}$) of the crossing events $H( \eta, R_N , \bullet)$ were trivial, in the sense that it converged to $0$ or $1$ as $N \to \infty$, then Theorem~\ref{noise} would itself be trivial.  However, this is not the case.  Further, one may deduce, using Theorem~\ref{varB}, that if $p > 0$ is sufficiently small, then with probability close to~$1$ (in~$\PB_{\lambda_c/p}$) the same is true for the model~$\PP^B_p$, see Proposition~\ref{Bprop}.  This fact will be a vital tool in our proof of the noise sensitivity of this model, as it will allow us to bound the probability of the `one-arm event' (see Section~\ref{algsubsec}). Throughout $R_{a \times b}$ denotes the rectangle with side lengths $a$ and $b$, centred at the origin.

\begin{thm}[Alexander~\cite{Alex}]\label{nontriv}
For every $t>0$ there exists $c = c(t) > 0$ such that
$$c \, \le \, \PB_{\lambda_c}\Big( H \big( \eta, R_{N \times tN},\bullet \big)\Big) \, \le \, 1 - c$$
for every $N \in \N$.
\end{thm}

Theorem~\ref{nontriv} is in fact a slight extension of~\cite[Theorem~3.4]{Alex}, but it follows by the same argument. For completeness, we shall sketch the proof in Section~\ref{nontrivsec}.  From this bound, together with Theorem~\ref{varB}, we shall deduce the following bound (see Section~\ref{T1sec}).

\begin{prop}\label{Bprop}
For every $t,\gamma > 0$ there exist constants $c' = c'(t) > 0$ and $p_0=p_0(t,\gamma) > 0$ such that if $p\in(0,p_0)$, then
$$\PB_{\lambda_c/p}\Big( \Pr^B_p\big( H\big( \eta, R_{N \times tN}, \bullet \big) \big) \not\in(c',1-c') \Big) \, < \, \gamma$$
for every sufficiently large $N \in \N$.
\end{prop}

We next turn to one of the key new techniques we introduce in this paper.

\subsection{A new method for proving results for biased product measures}\label{hfsec}

We outline here a new method for deducing results in the setting of a density $p$ product measure from the uniform case (i.e., the case $p=1/2$).  The idea is the following.  Rather than considering directly the function $f \colon \{0,1\}^n \to [0,1]$ where $\{0,1\}^n$ is endowed with density $p$ product measure, we consider a related function $h_f \colon \{0,1\}^n \to [0,1]$ where $\{0,1\}^n$ is endowed with uniform measure.  This function $h_f$ is obtained from $f$ by a smoothing (or averaging) operation.
Assume for now that $p\le 1/2$.
Given any function $f : \{0,1\}^n \to [0,1]$ we define $h_f$ as
\begin{equation}\label{def:h}
h_f(X) \, := \, \E\big[ f(Z) \,\big|\, X \big],
\end{equation}
where $Z\in\{0,1\}^n$ is obtained as a $2p$-subset of $X$. More formally, independently of $X$, pick a $2p$-subset $Y$ of $[n]$ (i.e., $\P(Y_i = 1) = 2p$ for each $i \in [n]$, all independently) and define $Z$ as the coordinate wise product of $X$ and $Y$, i.e.,
\begin{equation}\label{def:XYZ}
Z_i:=X_iY_i\quad\text{for each }i\in[n].
\end{equation}

Observe that $h_f: \{0,1\}^n \to [0,1]$ and that if $X$ is uniformly chosen, then $Z$ is a $p$-subset of $[n]$; Assume from now on that $X$ is uniformly chosen in $\{0,1\}^n$. By relating various parameters of $f(Z)$ and $h_f(X)$,
results about one may be deduced from results concerning the other. The connection between $f$ and $h_f$ is given by the following proposition. In the interest of generality we state the proposition for all $p\in (0,1)$.  For $p>1/2$ we define $Z$ as a $p$-subset containing $X$.  Formally, we set $Z_i = 1-(1-X_i)Y_i$ for each $i$, where $Y$ is a $2(1-p)$-subset of $[n]$. Following the standard convention, by $f$ being {\em monotone} we mean that $f(\omega)\le f(\omega')$ for every $\omega,\omega' \in\{0,1\}^n$ such that $\omega_j \le \omega'_j$ for every $j\in[n].$

\begin{prop}\label{prop:hf}
Let $f \colon \{0,1\}^n \to [0,1]$ and $p \in(0,1)$, and set $\bar{p}=\min\{p,1-p\}$.
\begin{itemize}
\item[$(i)$] If $f$ is monotone then $h_f$ is monotone.\\[-1.5ex]
\item[$(ii)$] $\Inf_{1/2,i} (h_f) \, \le \, 2\bar{p}\, \Inf_{p,i}(f)$, and moreover equality holds if $f$ is monotone.\\[-1.5ex]
\item[$(iii)$] $(f_n)_{n \ge 1} \textup{ is NS}_p \; \Leftrightarrow \; \big(h_{f_n}\big)_{n \ge 1}\textup{ is NS.}$\\[+0.5ex]
Moreover, if $p \neq 1/2$ then this is also equivalent to $\ds\lim_{n \to \infty} \Var\big(h_{f_n}\big)=0$.
\end{itemize}
\end{prop}

We remark that the random variables $X$ and $Z$ realize the maximal coupling between the uniform measure and product measure of density $p$ on $\{0,1\}^n$, or similarly, between a uniformly chosen subset and a $p$-subset of $[n]$.

Other reduction methods have previously been used for similar purposes. See e.g.~\cite{BKKKL,Fried,Kell,KMS} for results in this direction.

\subsection{The deterministic algorithm method}\label{algsubsec}

In order to prove that $(f^B_N)_{N\ge 1}$ is NS$_p$ for $\PB_{\lambda_c/p}$-almost every $B$ (Theorem~\ref{BpNS}), we shall use the `algorithm method', which was also introduced by Benjamini, Kalai and Schramm~\cite{BKS} in the case $p = 1/2$. (We would like to thank Jeff~Steif for pointing out to us that the approach in~\cite{BKS} can be synthesized in the way it is presented here.)  Given a function $f \colon \{0,1\}^n \to [0,1]$, let $\A^*(f)$ denote the collection of deterministic algorithms which determine $f$.\footnote{An algorithm is simply a rule which, given the information about $\omega$ received so far, tells you which bit of $\omega$ to query next. It determines $f$ if it determines $f(\omega)$ for any input $\omega \in \{0,1\}^n$.}

\begin{defn}[Revealment of an algorithm]
Let $f \colon \{0,1\}^n \to [0,1]$ and let $\A \in \A^*(f)$. For each $p \in (0,1)$ and $j\in [n]$, define
$$\delta_j(\A,p) \, := \, \Pr_p\big( \A \textup{ queries bit $j$ when determining $f(\omega)$}\big),$$
where $\omega\in\{0,1\}^n$ is chosen according to $\Pr_p$.
The \emph{revealment} $\delta_K(\A,p)$ of $\A$ with respect to a set $K\subset[n]$ is defined to be $\max\{ \delta_j(\A,p) : j \in K \}$.
\end{defn}

Using Theorem~\ref{BKSp}, we shall prove the following theorem, which generalizes the method of~\cite{BKS} to the non-uniform set-up. 

\begin{thm}\label{algthm}
Let $r\in \N$ be fixed, and let $(f_n)_{n \ge 1}$ be a sequence of monotone functions $f_n: \{0,1\}^n \to [0,1]$.  For each $n \in \N$, let $\A_1,\ldots,\A_r \in \A^*(f_n)$ and let $[n] = K_1 \cup \ldots \cup K_r$. If, for a fixed $p \in (0,1)$, we have
$$
\delta_{K_i}(\A_i,p) \big( \log n \big)^6 \; \to \; 0
$$
as $n \to \infty$ for each $i \in [r]$, then $(f_n)_{n \ge1}$ is \textup{NS}$_p$.
\end{thm}

We aim to apply Theorem~\ref{algthm} to deduce noise sensitivity in the (discrete) model $\PP^B_p$. For this we shall need to define a deterministic algorithm which determines $f^B_N$, and show that it has low revealment for most sets $B$ (with respect to $\PB_{\lambda_c/p}$). The algorithm which we shall use is the continuum analogue of that used in~\cite{BKS}. Let $\eta$ be a $p$-subset of $B$. Roughly speaking, we `pour water' into the left-hand side of the square $R_N$, and allow water to infiltrate the occupied space $D(\eta)$. Thus, an element in $B$ will be queried only if it becomes wet via a path in $D(\eta)$ reaching from the left-hand side. For elements in the left half of $R_N$ we pour water into the right-hand side (see Section~\ref{T1sec} for a precise definition).

It is easy to see that the probability that an element $x \in B$ is queried by $\A$ is at most the probability of the corresponding `one-arm event', i.e., the event that there is a path in $D(\eta)$ from $x$ to the boundary of a square centered therearound (for background on arm-events, see e.g.~\cite{BR}). In the original Poisson Boolean model, a bound on this probability can be deduced from Theorem~\ref{nontriv}.
However, in order to apply Theorem~\ref{algthm} we need a bound for the model $\PP^B_p$; we obtain such a bound using Proposition~\ref{Bprop}.

In order to apply Proposition~\ref{Bprop}, we simply surround each point $x \in B\cap R_N$ by $c \log N$ disjoint annuli, and show that, with \emph{very} high probability (in $\PB_{\lambda_c/p}$), at least half of them are `good', in the sense that the probability (in $\Pr^B_p$) that there is a vacant loop around $x$ is at least $c''$, for some small constant $c'' > 0$. It will then follow that (for $\PB_{\lambda_c/p}$-almost every $B$), every $x \in B\cap R_N$ has probability at most $N^{-\delta}$ (in $\Pr^B_p$) of being queried by $\A$, when $N$ is large (see Section~\ref{T1sec}).

\subsection{Hypergraphs}\label{hypersubsec}

Theorem~\ref{varB} provides a very general bound on the variance that arises in settings where two stages of randomness are used to select a random subset. The main step in the proof of Theorem~\ref{varB} is to prove a variance bound (Proposition~\ref{var}) for the case where the random sets $A \subset B$ have fixed sizes $m\le k$.  It is then relatively straightforward to deduce a corresponding bound on $\Var_{q}\big( r_{\HH}(B,p) \big)$, and thus prove Theorem~\ref{varB}, by bounding other factors that might contribute towards the variance.  These bounds are obtained using Chernoff's inequality (see Section~\ref{hypersec}).

We shall control $\Var\big( X_m(B_k) \big)$, where $B_k$ is a uniformly chosen $k$ element subset of $[n]$ and $X_m(B)=X_m(B,\HH)$ counts the number of hypergraph edges of size $m$ contained in $B\subset [n]$, using the following theorem of Bey~\cite{Bey} concerning the sum of squares of degrees in hypergraphs. It generalized results of Ahlswede and Katona~\cite{AK} and de Caen~\cite{dC}, and answered a question of Aharoni~\cite{Aha}.

Let $e(\HH)$ denote the number of edges in a hypergraph $\HH$, and, given a set $T \subset [n]$, let $d_\HH(T)$ denote the \emph{degree} of $T$ in $\HH$, i.e., the number of edges of $\HH$ which contain $T$. The following result bounds the sum of the squares of the degrees over sets of size $t$ in an $m$-uniform hypergraph, i.e., one in which all edges have size $m$. By convention, we let ${n\choose k}:=0$ for $k<0$ and $k>n$.

\begin{Bey}[Bey~\cite{Bey}]
Let $\HH$ be an $m$-uniform hypergraph on $n$ vertices, and let $t \in [m]$. 
Then
$$d_2\big( \HH,t \big) \; := \; \sum_{T \subset [n],\, |T| = t} d_\HH(T)^2 \; \le \; \frac{ {m \choose t} { {m-1} \choose t} }{ {{n-1} \choose t} } e(\HH)^2 \,+\, {{m-1} \choose {t-1}} {{n-t-1} \choose {m-t}} e(\HH).$$
\end{Bey}

To see how Bey's inequality is related to the variance of $X_m(B_k)$, observe that $d_\HH(T)^2$ counts the number of (ordered) pairs of edges of size $m$ in $\HH$ which both contain $T$. Thus, summing over $t$ (with appropriate weights), we obtain an upper bound on $\Ex\big( X_m(B_k)^2 \big)$.

\section{Non-triviality of the crossing probability at criticality}\label{nontrivsec}

In this section we shall sketch the proof (from~\cite{Alex}) of Theorem~\ref{nontriv}, which says that at criticality, the probability of crossing a rectangle is bounded away from zero and one. The proof is based on the RSW Theorem for the Poisson Boolean model, which was proved by Roy~\cite{Roy} for the vacant space (see below), and by Alexander~\cite{Alex} for the occupied space.

Recall that $\PB_\lambda$ indicates that the configuration $\eta\subset\RR^2$ is chosen according to a Poisson process with intensity $\lambda$.
Let $V\big( \eta, R, \circ \big)$ denote the event that there is a vacant vertical crossing of $R$, i.e., a crossing using only points of $R_N \setminus D(\eta)$, and define $V\big( \eta, R, \bullet \big)$ (vertical crossing in $D(\eta) \cap R_N$) and $H\big( \eta, R, \circ \big)$ similarly.

\begin{RSWvac}[Roy~\cite{Roy}, see Theorem~4.2 of~\cite{MR}]\label{RSWvac}
For every $\delta, t, \lambda > 0$, there exists an $\eps = \eps(\delta,t, \lambda) > 0$ such that the following holds for every $a,b,c > 0$ with $c \le 3a/2$. If
$$\PB_{\lambda}\Big( H\big( \eta,R_{a \times b},\circ \big) \Big) \,\ge\, \delta,\quad \text{and} \quad\PB_\lambda\Big(H \big( \eta, R_{b \times c},\circ \big) \Big) \,\ge\, \delta,$$
then $\PB_\lambda\Big( H\big( \eta,R_{ta \times b},\circ \big) \Big) \ge \eps$.
\end{RSWvac}

We remark that this result was in fact proved in substantially greater generality: it holds for random radii, with arbitrary distribution on $(0,r)$ (where $r \in \RR_+$ is arbitrary). 
Alexander~\cite{Alex} proved the corresponding statement for the occupied space (for fixed radii), and used this result to prove the following characterization.

\begin{thm}[Theorem~3.4 of~\cite{Alex}]\label{thm:Alex}
In the Poisson Boolean model, there exists $\theta > 0$ such that the following are equivalent:
\begin{enumerate}
\item[$(a)$] There is almost surely an infinite occupied component.\\[-1.5ex]
\item[$(b)$] $\ds\lim_{N \to \infty}\PB_\lambda\Big( H \big( \eta,R_N,\bullet \big) \Big) = 1$.\\[-1.5ex]
\item[$(c)$] $\ds\lim_{N \to \infty}\PB_\lambda\Big( H \big( \eta,R_{3N \times N}, \bullet \big) \Big) = 1$.\\[-1.5ex]
\item[$(d)$] There exists $N \in \N$ such that $\PB_\lambda\Big( H \big( \eta,R_{3N \times N},\bullet \big)  \Big) > 1 - \theta$.
\end{enumerate}
The same holds true if `occupied' is changed for `vacant' throughout.
\end{thm}

It follows immediately that there is no percolation at criticality for either the occupied or vacant space.

\begin{cor}[Corollary~3.5 of~\cite{Alex}]\label{Cor:Alex}
At $\lambda = \lambda_c$, there is almost surely no infinite component in the occupied space $D(\eta)$,  and no infinite component in the vacant space $\RR^2 \setminus D(\eta)$.
\end{cor}

\begin{proof}
 A standard argument shows that $\PB_\lambda\big( H \big( \eta,R_{3N \times N},\bullet \big) \big)$ is a continuous function of $\lambda$, and so the set of $\lambda \in \RR$ for which property $(d)$ of Theorem~\ref{thm:Alex} holds is an open set.
\end{proof}

Theorem~\ref{nontriv} follows immediately from Corollary~\ref{Cor:Alex}, together with the following slight extension of Theorem~\ref{thm:Alex}.

\begin{thm}\label{Alex:ext}
Let $\eta$ be a subset of $\RR^2$ chosen according to a Poisson point process with intensity $\lambda$. Then, for every $t > 0$,
\begin{itemize}
\item[$(a)$] $\ds\sup_{N \ge 1} \, \PB_\lambda\Big( H \big( \eta,R_{N \times tN},\bullet \big) \Big) = 1 \quad \Rightarrow \quad D(\eta)$ percolates almost surely.
\item[$(b)$] $\ds\sup_{N \ge 1} \, \PB_\lambda\Big( H \big( \eta,R_{N \times tN}, \circ \big) \Big) = 1 \quad \Rightarrow \quad \RR^2 \setminus D(\eta)$ percolates almost surely.
\end{itemize}
\end{thm}

The proof of Theorem~\ref{Alex:ext} is almost identical to that of Theorem~\ref{thm:Alex}; for completeness, we shall sketch the argument.

\begin{proof}[Sketch proof of Theorem~\ref{Alex:ext}]
We shall prove only $(a)$; part $(b)$ follows by the same proof, except using the Occupied RSW Theorem~\cite[Theorem~2.1]{Alex} in place of the Vacant RSW Theorem. We claim that our assumption implies property $(d)$ in Theorem~\ref{thm:Alex}, and hence (by property $(a)$ of the theorem) that $D(\eta)$ percolates. We remark that the implication $(d) \Rightarrow (a)$ in Theorem~\ref{thm:Alex} follows by a straightforward Peierls-type argument.

We want to show that property $(d)$ of Theorem \ref{thm:Alex} holds. It is not hard to show that
$$\ds\sup_{N \ge 1} \, \PB_\lambda\Big( H \big( \eta,R_{N \times tN},\bullet \big) \Big) = 1 \quad \Rightarrow \quad \ds\sup_{N \ge 1} \, \PB_\lambda\Big( H \big( \eta,R_{N \times N},\bullet \big) \Big) = 1,$$
by the Vacant RSW Theorem, applied with $a=b=c=N$. The latter implies that for each $\eps > 0$ there exists $N = N(\eps) \ge 1$ such that
\be\label{eq:upperassumption}
\PB_\lambda\Big( H\big(\eta, R_{N \times N}, \bullet \big) \Big) \; = \; 1 \,-\, \PB_\lambda\Big( V\big( \eta, R_{N \times N}, \circ \big) \Big) \; > \; 1 - \eps.
\ee
Next, observe that for every $N > 0$ and $k \in \N$,
$$\PB_\lambda\Big( V\big( \eta, R_{3N \times N},\circ \big) \Big) \; \le \; \big( 2k + 1 \big) \PB_\lambda \Big( V\big( \eta, R_{N \times N},\circ \big) \Big) \,+\, 2k \cdot \PB_\lambda\Big( H\big( \eta, R_{\frac{k-1}{k}N \times N},\circ \big) \Big),$$
and moreover that
$$\PB_\lambda\Big( V\big( \eta, R_{3N \times N},\circ \big) \Big) \; \le \; 2k \cdot \PB_\lambda\Big( V\big( \eta, R_{\frac{k+1}{k} N \times N}, \circ \big) \Big) \,+\, \big( 2k - 1 \big) \PB_\lambda\Big( H \big( \eta, R_{N \times N}, \circ \big) \Big).$$
To see these, partition the rectangle $R_{3N \times N}$ into $\frac{N}{k} \times N$ rectangles $B_1,\ldots,B_{3k}$, and consider the leftmost and rightmost pieces $B_j$ touched by a vertical path across $R_{3N \times N}$. The first follows because either an $N \times N$ square (made up of $B_j$'s) is crossed vertically, or a $\left( \frac{k-1}{k} \right) N \times N$ rectangle is crossed horizontally. The second follows because either a $\frac{k+1}{k} N \times N$ rectangle is crossed vertically, or an $N \times N$ square is crossed horizontally.

Thus, either $\PB_\lambda\Big( V\big( \eta, R_{3N \times N},\circ \big) \Big)$ can be made arbitrarily small, as required, or there exists $\delta > 0$ such that
\be\label{eq:PHVkn}
\PB_\lambda\Big( H \big( \eta, R_{\frac{k-1}{k}N \times N}, \circ \big) \Big) \,\ge\, \delta \quad \text{and} \quad \PB_\lambda\Big( V\big( \eta, R_{\frac{k+1}{k} N \times N}, \circ \big) \Big) \,\ge\, \delta
\ee
for every $N = N(\eps)$ and every $\eps > 0$.

Now, apply the Vacant RSW Theorem with $a = \frac{k-1}{k} N$, $b = N$, $c = \frac{k+1}{k} N$, for some $N > 0$. Note that $c \le 3a/2$ if $2(k+1) \le 3(k-1)$, which holds if $k \ge 5$. Setting $t = \frac{k}{k-1}$, it follows that if~\eqref{eq:PHVkn} holds for $N$, then
\be\label{eq:PHnxn}
\PB_\lambda\Big( H\big(\eta, R_{N \times N},\circ \big) \Big) \; \ge \; \eps'.
\ee
where $\eps' = \eps(\delta,t,\lambda) > 0$ is given by the Vacant RSW Theorem.

Hence if~\eqref{eq:PHVkn} holds for $N = N(\eps')$ then~\eqref{eq:PHnxn} also holds, and~\eqref{eq:PHnxn} contradicts~\eqref{eq:upperassumption}. Thus~\eqref{eq:PHVkn} must fail to hold for $N = N(\eps')$, and so, by the observations above, $\PB_\lambda\big( V\big( \eta, R_{3N \times N},\circ \big) \big)$ can be made arbitrarily small, as required.
\end{proof}

\section{BKS Theorem for biased product measures} \label{BKSsec}

A tool that has turned out to be very useful in connection with the study of Boolean functions is discrete Fourier analysis. For $\omega\in\{0,1\}^n$ and $i\in [n]$, we define
\[
\chi_i^p(\omega)=\left\{
\begin{array}{cc}
-\sqrt{\frac{1-p}{p}} & \textrm{if } \omega_i=1 \\
\sqrt{\frac{p}{1-p}} & \textrm{otherwise. }
\end{array}
\right.
\]
Furthermore, for $S\subset [n],$ let $\chi_S^p(\omega):=\prod_{ i \in S } \chi_i^p(\omega)$. (In particular, $\chi_\emptyset^p$ is the constant
function 1.) We observe that for $i \neq j$
$$\E_p\big[\chi_i^p(\omega)\chi_j^p(\omega)\big] \; = \; \left( \frac{1-p}{p} \right) p^2 \,+\, \left( \frac{p}{1-p} \right) (1-p)^2 \, - \, 2p(1-p) \; = \; 0.$$
In fact, it is easily seen that the set $\{\chi_S^p\}_{S\subset [n]}$ forms an orthonormal basis for the set of functions $f:\{0,1\}^n\mapsto \RR$. We can therefore express such functions using the so-called \emph{Fourier-Walsh representation} (see~\cite{Paley,Walsh}):
\begin{equation}\label{FWeq}
f(\omega)=\sum_{S\subset [n]} \hat{f}^p(S)\chi_S^p(\omega),
\end{equation}
where $\hat{f}^p(S) := \E_p[ f \chi_S^p]$. 

The following lemma was proved in~\cite{BKS} in the uniform case; its generalization to arbitrary (fixed) $p$ is similarly straightforward.

\begin{lemma} \label{lem:NSequiv}
Let $p \in (0,1)$, and let $(f_n)_{n\ge1}$ be a sequence of functions $f_n \colon \{0,1\}^n\mapsto [0,1]$. The following two conditions are equivalent.
\begin{itemize}
\item[$(i)$] The sequence $(f_n)_{n\ge1}$ is \textup{NS}$_p$.\\[-2ex]
\item[$(ii)$] For every $k \in \N$,
\[
\lim_{n\to\infty} \sum_{0 < |S| \le k}\hat{f_n}^p(S)^2 = 0.
\]
\end{itemize}
\end{lemma}

Although our results hold for arbitrary $p \in (0,1)$, we shall prove them only for $p \le 1/2$, since this is the case we shall need in our applications. The proofs for $p > 1/2$ all follow in exactly the same way. From now on, we will not stress that $S\subset [n]$ in the notation. Furthermore, when $p = 1/2$ we shall write $\chi_S$ for $\chi_S^p$ and $\hat{f}(S)$ for $\hat{f}^p(S)$.

\begin{proof}[Proof of Lemma~\ref{lem:NSequiv}]
Note that $\Ex_p\big[ \chi_S^p(\omega) \chi_{S'}^p(\omega^\eps) \big] = 0$ if $S \neq S'$, that $\E_p\big[f_n(\omega)\big] = \hat{f_n}^p(\emptyset)$, and that $$\E_p\big[\chi_S^p(\omega)\chi_S^p(\omega^\eps)\big] = (1-\eps)^{|S|},$$
since this is zero whenever at least one of the coordinates $\{ \omega_i : i \in S \}$ is re-randomized, and one otherwise. By~\eqref{FWeq}, it follows that
\begin{eqnarray*}
&& \E_p\big[f_n(\omega) f_n(\omega^{\eps}) \big] - \E_p\big[f_n(\omega)\big]^2 \; =\; \E_p\left[\sum_S\hat{f_n}^p(S)\chi_S^p(\omega) \sum_{S'} \hat{f_n}^p(S')\chi_{S'}^p(\omega^{\eps})\right]-\hat{f}_n^p(\emptyset)^2\\
&&\qquad \qquad= \; \sum_{S \neq \emptyset } \hat{f_n}^p(S)^2 \E_p\big[\chi_S^p(\omega)\chi_S^p(\omega^\eps)\big] \; = \; \sum_{S\neq \emptyset}\hat{f_n}^p(S)^2(1-\eps)^{|S|},
\end{eqnarray*}
from which both implications follow easily.
\end{proof}

Recall from~\eqref{def:h} that we define $h_f(X) := \E[ f(Z) | X ]$, where $X,Y \in \{0,1\}^n$ are independent random variables, $X$ is chosen uniformly, $Y$ is a $2p$-random set, and $Z_i = X_i Y_i$ for every $i\in [n]$. The key fact, that the sequence $(f_n)_{n\ge1}$ is NS$_p$ if and only if $(h_{f_n})_{n\ge1}$ is NS, will follow directly from Lemma \ref{lem:NSequiv}, together with the following result.

\begin{prop}\label{prop:hfspectrum}
Let $f \colon \{0,1\}^n \to [0,1]$ and $p \in (0,1)$, and set $\bar{p}=\min\{p,1-p\}$. Then, for every $S \subset [n]$,
\[
\hat{h}_f(S) \,=\, \left(\frac{\bar{p}}{1-\bar{p}}\right)^{|S|/2}\hat{f}^p(S)
\]
\end{prop}

\begin{proof}
We shall prove the proposition in the case $p \le 1/2$; the other case follows similarly. Let $f \colon \{0,1\}^n \to [0,1]$ and $S \subset [n]$. By the definitions, we have
\begin{equation} \label{eqn9}
\begin{aligned}
\hat{h}_f(S) &\; =\; \E\big[h_f(X)\chi_S(X) \big] \; = \; \E\Big[ \E\big[f(Z) \,\big|\, X\big] \chi_S(X) \Big]\\[+0.5ex]
& \;=\; \E\Big[ \E\big[f(Z)\chi_S(X) \,\big|\, X\big] \Big] \; = \; \E\big[ f(Z)\chi_S(X) \big] \; = \; \E\Big[ f(Z)\E\big[\chi_S(X) \,\big|\, Z\big] \Big].
\end{aligned}
\end{equation}
Furthermore, $Z_i=1$ implies $X_i=1$, which implies $\chi_i(X) = -1$, so
\[
\E\big[ \chi_i(X) \,\big|\, Z_i = 1 \big] \,=\, -1,
\]
while $Z_i = 0$ and $X_i = 1$ implies that $Y_i = 0$, so
\begin{eqnarray*}
\E\big[\chi_i(X) \big| Z_i = 0\big] & = & 1 \,-\, 2 \cdot \P\big( X_i=1 \,\big|\, Z_i=0 \big)\\
& = & 1 \,-\, 2 \cdot \frac{\P(Z_i=0 \,|\, X_i=1)\P( X_i=1)}{\P(Z_i=0)} \; = \; 1 - \frac{1-2p}{1-p} \; = \; \frac{p}{1-p}.
\end{eqnarray*}
We conclude that $\E\big[\chi_i(X) \big| Z_i\big] = \sqrt{\frac{p}{1-p}}\chi_i^p(Z)$. Therefore, since the $X_i$ and $Y_i$ are all independent,
\begin{eqnarray*}
\E\big[ \chi_S(X) \,\big|\, Z \big] & = & \prod_{i\in S} \E\big[ \chi_i(X) \,\big|\, Z_i \big] \; = \; \prod_{i\in S} \sqrt{\frac{p}{1-p}}\chi_i^p(Z) \; = \; \left(\frac{p}{1-p}\right)^{|S|/2}\chi_S^p(Z).
\end{eqnarray*}
Inserting this into~\eqref{eqn9} gives the result.
\end{proof}

It is now straightforward to deduce Proposition~\ref{prop:hf} from Proposition~\ref{prop:hfspectrum}.

\begin{proof}[Proof of Proposition \ref{prop:hf}]
We shall assume that $p \le 1/2$; once again, the other case follows similarly. Let $f \colon \{0,1\}^n \to [0,1]$.

$(i)$ Suppose that $f$ is monotone; we claim that $h_f$ is also monotone. Indeed, observe that
\begin{equation}\label{25i}
h_f(X) \, = \, \E\big[f(Z) \big|X\big] \, = \, \E\big[f(XY) \big| X\big] \, = \, \sum_{\xi \in \{0,1\}^n}f(X\xi) \P(Y=\xi).
\end{equation}
But if $f$ is monotone, then $f(X\xi)$ is also monotone in $X$ for every $\xi \in \{0,1\}^n$. Thus~\eqref{25i} implies that $h_f$ is monotone, as required.

$(ii)$ We next claim that $\Inf_{1/2,i} (h_f) \le 2p \cdot  \Inf_{p,i}(f)$ for every $i \in [n]$. For every $i \in [n]$ and $k \in \{0,1\}$, let $X^{i \to k} \in \{0,1\}^n$ be defined by $X_j^{i \to k} = X_j$ if $j \neq i$, and $X_i^{i \to k} = k$.

By the definition, we have
\begin{equation}\label{25ii}
\Inf_{1/2,i}(h_f) \; = \; \E\big[\big|h_f(X) - h_f(\sigma_iX)\big|\big] \; = \; \E\Big[\Big|\E\big[f(Z) \big| X^{i \to 1}\big]-\E\big[f(Z) \big| X^{i \to 0}\big]\Big|\Big].
\end{equation}
Now, if $X_i = 1$, then $Y_i = 1$ if and only if $Z_i = 1$, and if $X_i = 0$ then $Z_i = 0$, so the right-hand side of~\eqref{25ii} is equal to
$$\E\bigg[ \Big| 2p \cdot  \E\big[f(Z) \,\big|\, X^{i \to 1}, Z_i = 1 \big]\, +\, \big( 1 - 2p \big) \E\big[f(Z) \,\big|\, X^{i \to 1}, Z_i = 0 \big] \,-\, \E\big[ f(Z) \,\big|\, X^{i \to 0}, Z_i = 0 \big] \Big| \bigg].$$
But given $Z_i$, the value of $X_i$ is irrelevant to $f(Z)$, so we have (with obvious notation $X_{\{i\}^c}$)
\begin{eqnarray}
\Inf_{1/2,i}(h_f) & = & 2p\, \E\Big[ \left| \E\big[f(Z^{i \to 1}) \,-\, f(Z^{i \to 0}) \,\big|\, X_{\{i\}^c} \big] \right| \Big] \nonumber \\[+0.5ex]
& \le & 2p\, \E\Big [\big| f(Z^{i \to 1}) \,-\, f(Z^{i \to 0}) \big| \Big]\; = \; 2p\,\Inf_{p,i}(f),\label{25ii2}
\end{eqnarray}
as required. Finally, note that the inequality in~\eqref{25ii2} be replaced by an equality when~$f$~is monotone.

$(iii)$ We are required to show that $(f_n)_{n\ge 1}$ is NS$_p$ if and only if $(h_{f_n})_{n\ge 1}$ is NS. Indeed, by Lemma~\ref{lem:NSequiv}, $(f_n)_{n\ge 1}$ is NS$_p$ if and only if $\sum_{0 < |S| \le k}\hat{f_n}^p(S)^2 \to 0$ as $n \to \infty$ for every fixed $k$, and by Proposition \ref{prop:hfspectrum},
$$\lim_{n \to \infty}\sum_{0<|S|\le k}\hat{h}_{f_n}(S)^2=0\quad \Leftrightarrow \quad\lim_{n \to \infty} \sum_{0 < |S| \le  k} \hat{f_n}^p(S)^2 = 0$$
for every such $k$. But by Lemma~\ref{lem:NSequiv} (applied with $p = 1/2$), we have that $(h_{f_n})_{n\ge 1}$ is NS if and only if $\sum_{0 < |S| \le k} \hat{h}_{f_n}(S)^2 \to 0$ as $n \to \infty$ for every fixed $k$, so the result follows.

Finally, note that $\Var(h_{f_n}) = \sum_{S\neq\emptyset}\hat{h}_{f_n}(S)^2$. Thus, by Proposition~\ref{prop:hfspectrum}, if $p \ne 1/2$ then
$$\Var(h_{f_n}) \; = \; \sum_{S\neq\emptyset} \left(\frac{p}{1-p}\right)^{|S|} \hat{f_n}^p(S)^2 \; \to \; 0$$
as $n \to \infty$ if and only if $\sum_{0 < |S| \le k} \hat{f_n}^p(S)^2 \to 0$ as $n \to \infty$ for every fixed $k$, as claimed.
\end{proof}

The BKS Theorem for biased product measures follows almost immediately from the uniform case, together with Proposition~\ref{prop:hf}.

\begin{proof}[Proof of Theorem~\ref{BKSp}]
Let $(f_n)_{n \ge 1}$ be a sequence of functions $f_n \colon \{0,1\}^n\to[0,1]$, let $p\in(0,1)$, and assume that $I\!I_p(f_n) \to 0$ as $n \to \infty$. We are required to show that $(f_n)_{n\ge1}$ is NS$_p$.

By Proposition~\ref{prop:hf}$(ii)$, we have
$$I\!I(h_{f_n}) \; \le \; 4\bar{p}^2 \cdot I\!I(f_n),$$
and so $I\!I(h_{f_n}) \to 0$ as $n \to \infty$. By the BKS Theorem (i.e., Theorem~\ref{BKSp} in the case $p = 1/2$), which was proved in~\cite{BKS}, it follows that $(h_{f_n})_{n \ge 1}$ is NS.

But, by Proposition~\ref{prop:hf}$(iii)$, we have $(h_{f_n})_{n\ge1}$ is NS if and only if $(f_n)_{n\ge1}$ is NS$_p$. Hence $(f_n)_{n\ge1}$ is NS$_p$, as required.
\end{proof}

\section{The deterministic algorithm approach}\label{Algsec}

In this section we shall prove Theorem \ref{algthm}, the uniform case of which was proved in~\cite{BKS}. We shall use the method of~\cite{BKS}, together with Theorem~\ref{BKSp} and some of the results from the previous section. 

We need the following definition.
\begin{defn}[The Majority function]\label{def:MajK2}
For every $K\subset[n]$ let $M_K:\{0,1\}^n \to \{-1,0,1\}$ be defined by
\bea
M_K(X):=\left\{
\begin{aligned}
1 & \quad\text{if }\sum_{i\in K}(2X_i-1) > 0\\
0 & \quad\text{if }\sum_{i\in K}(2X_i-1) = 0\\
-1 & \quad\text{if }\sum_{i\in K}(2X_i-1) < 0.
\end{aligned}
\right.
\eea
\end{defn}

Throughout this section let $X$ and $Z$ be the random variables defined in Section~\ref{hfsec}, so $X \in \{0,1\}^n$ is chosen uniformly, and $Z \in \{0,1\}^n$ with $Z_i = X_i Y_i$, where $Y_i$ is chosen according to product measure with density $2p$. (We assume again for simplicity that $p \le 1/2$.)

Theorem~\ref{algthm} will follow by combining the BKS Theorem with the following two propositions, of which in~\cite{BKS} the first was proved in the case $p=1/2$, and the second was proved for $p=1/2$ for functions taking values in $\{0,1\}$. We shall generalize them to the biased setting.

\begin{prop}\label{prop:InfMKcorr}
There exists a constant $C > 0$ such that, if $f \colon \{0,1\}^n \to [0,1]$ is monotone, $p \in (0,1)$ and $K \subset [n]$, then
$$\sum_{j \in K} \Inf_{p,j}(f) \; \le \; \frac{C}{\min\{p,1-p\}} \sqrt{|K|} \E\big[ f(Z)M_K(X) \big] \left( 1 + \sqrt{- \log\E\big[ f(Z) M_K(X) \big]} \right).$$
\end{prop}

Recall that the revealment $\delta_K(\A,p)$ of an algorithm with respect to a set $K \subset [n]$ is defined to be $\max\{ \delta_j(\A,p) : j \in K \}$, where $\delta_j(\A,p) = \Pr_p\big( \A \textup{ queries coordinate $j$}\big)$.

\begin{prop}\label{prop:algMKcorr}
There exists a constant $C > 0$ such that, if $f \colon \{0,1\}^n \to [0,1]$, $p \in (0,1)$, $K \subset [n]$
and $\A\in \A^*(f),$ then
$$\E\big[ f(Z) M_K(X) \big] \,\le\, C\, \delta_K(\A,p)^{1/3} \log n.$$
\end{prop}

We begin by proving Proposition~\ref{prop:InfMKcorr}, which follows almost immediately from the uniform case, together with Proposition~\ref{prop:hf}.

\begin{proof}[Proof of Proposition \ref{prop:InfMKcorr}]
The proposition was proved in~\cite[Corollary~3.2]{BKS} in the case $p = 1/2$; we apply this result to the function $h_f$. It follows that
$$\sum_{j \in K} \Inf_{1/2,j}(h_f) \,\le\, C\sqrt{|K|}\E\big[h_f(X)M_K(X)\big]\left(1+\sqrt{-\log\E\big[h_f(X)M_K(X)\big]}\right).$$
for some $C > 0$. Next, observe that
$$\E\big[h_f(X)M_K(X)\big] \,=\, \E\Big[ \E\big[ f(Z) M_K(X) \,\big|\, X \big] \Big] \,=\, \E\big[ f(Z) M_K(X) \big].$$
Since $f$ is monotone, we have $\Inf_{1/2,j}(h_f) = 2\bar{p} \cdot \Inf_{p,j}(f)$, by Proposition~\ref{prop:hf}, and so the result follows.
\end{proof}

The proof of Proposition \ref{prop:algMKcorr} will be based on the argument used in~\cite[Section~4]{BKS}, but modified to fit in the current setting. The strategy is roughly as follows: let $V$ denote the set of coordinates queried by the algorithm. Then with high probability, $V \cap K$ is small enough so that $M_K(X)$ will (probably) be determined by the values of bits of $X$ in $K \setminus V$. By a careful coupling, we can make these independent of the value of $f$, and thus $\E\big[f(Z)M_K(X)\big]$ is small.

We shall use Chernoff's inequality; see, e.g., \cite[Appendix A]{AS}. Let $\Bin(n,p)$ denote the binomial distribution with parameters $n$ and $p$. Throughout the rest of the paper, $\xi_{n,p}$ will denote a binomially distributed random variable with parameters $n\in\N$ and $p\in(0,1)$.

\begin{Chernoff}
  \label{chernoff}
Let $n \in \N$ and $p \in (0,1)$, and let $a > 0$. Then
\begin{equation} \label{eqn1}
\Pr\Big( \big| \xi_{n,p} - pn \big| > a \Big) \; < \; 2\exp\left( -\frac{a^2}{4pn} \right)
\end{equation}
if $a \le pn/2$, and $\Pr\big( | \xi_{n,p} - pn | > a \big) < 2\exp\big( -pn/16 \big)$ otherwise.
If $p=1/2,$ then (\ref{eqn1}) holds for every $a \geq 0.$
\end{Chernoff}

We shall also use the following simple property of the binomial distribution, which follows by Stirling's formula.

\begin{obs}\label{binmax}
There exists $C > 0$ such that for any $n \in \N$, $p \in (0,1)$ and $a \in \N$, 
$$\P\big( \xi_{n,p} = a \big) \, \le \, \frac{C}{\sqrt{np(1-p)}}.$$
\end{obs}

We are now ready to prove Proposition~\ref{prop:algMKcorr}.

\begin{proof}[Proof of Proposition \ref{prop:algMKcorr}]
We assume as usual that $p \le 1/2$, and note that the proof for $p > 1/2$ is similar. Let $f \colon \{0,1\}^n\to [0,1]$ and $\emptyset \ne K \subset [n]$ (if $|K| = 0$ then both sides are zero).
We begin by defining our coupling; the purpose is to make the values of $X_i$ outside $V$ independent of those inside.

We shall obtain the random variables $X$ and $Z$, defined in~\eqref{def:XYZ}, as follows. Let $Z^1\in \{0,1\}^K$,  $Z^2 \in\{0,1\}^{[n] \setminus K}$ and $Z^3,Z^4 \in \{0,1\}^n$ be such that
$$\P\big(Z^j_i = 1\big) = p,$$
independently for each $i$ and $j$. Similarly, let $W^1\in\{0,1\}^K$, $W^2\in\{0,1\}^{[n]\setminus K}$ and $W^3\in\{0,1\}^n$ be independent of the $Z^j_i$, and such that
$$\P\big( W^j_i = 1 \big) \,=\, \frac{1-2p}{2(1-p)},$$
independently for every $i$ and $j$. Set $X_i^j = \max\big\{ Z_i^j, W_i^j \big\}$, and observe that
$$\P\big( X_i^j = 1 \big) \,=\, \P\big( Z_i^j  = 1 \big) + \P\big( Z_i^j = 0 \big) \P\big( W_i^j = 1 \big) \,=\, p + \frac{(1-p)(1-2p)}{2(1-p)} \,=\, \frac{1}{2}$$
for every $i$ and $j$.

Next, we describe how to use the $X_i^j$ and $Z_i^j$ to assign values to coordinates, depending on the order in which they are queried by $\A$. Indeed, run the algorithm, and do the following:
\begin{itemize}
\item[1.] If $j \in K$ is queried, and is the $k^{th}$ element of $K$ to have been queried by $\A$, then set $Z_j := Z_k^1$ and $X_j := X_k^1$.
\item[2.] If $j \not\in K$ is queried, and is the $k^{th}$ element of $[n]\setminus K$ to have been queried by $\A$, then set $Z_j := Z_k^2$ and $X_j := X_k^2$.
\item[3.] When the algorithm stops, let $\pi : K \setminus V \to [|K \setminus V|]$ be an arbitrary bijection, and for each $j \in K \setminus V$ set $Z_j := Z_k^3$ and $X_j := X_k^3$, where $k = \pi(j)$.
\item[4.] Finally, let $Z_j := Z_j^4$ and $X_j := X_j^4$ for each $j \in [n] \setminus (V \cup K)$.
\end{itemize}
Note that $X$ is chosen uniformly and $Z$ according to the product measure with density $p$. Moreover, note that if $Z_i = 1$ then $X_i = 1$, so the coupling is as in~\eqref{def:XYZ}, as claimed.

Let $V \subset [n]$ be the (random) set of coordinates which are queried by the algorithm, and note that $V$ is independent of $Z^3$ and $W^3$. We first show that the set $V \cap K$ is likely to be small. Indeed, we have
$$\E\big[ |V\cap K|\big] \,=\, \sum_{j \in K} \delta_j(\A,p) \,\le\, |K|\delta_K(\A,p),$$
and so, if we define
$$B_1 \, := \, \big\{ |V\cap K| \ge |K| \delta_K(\A,p)^{2/3} \big\},$$
then $\P(B_1) \le \delta_K(\A,p)^{1/3}$, by Markov's inequality.

Next we shall deduce that, with high probability, the difference between the number of 0s and 1s on $V \cap K$ is less than that on $K \setminus V$. Indeed, let $S_k := \sum_{j=1}^k (2X_j^1 - 1)$ denote this difference on the first $k$ coordinates of $X^1$, and let $T_k := \sum_{j=1}^{k} (2X_j^3 - 1)$ denote the same thing for $X^3$. Let
$$B_2 \, :=\, \Big\{ \exists \, k \le |K| \delta_K(\A,p)^{2/3} \,:\, |S_k| \ge \sqrt{|K|} \delta_K(\A,p)^{1/3} \log n \Big\},$$
and let
$$B_3 \, := \, \left\{ |T_{|K\setminus V|}| \le \sqrt{|K|} \delta_K(\A,p)^{1/3} \log n \right\}.$$

\medskip
\noindent \textbf{Claim:} $\Pr\big( B_1 \cup B_2 \cup B_3 \big) \, = \,  O\left( \delta_K(\A,p)^{1/3} \log n \right)$.
\medskip

Before proving the claim, let's see how it implies the proposition. Set $Q = \big( B_1 \cup B_2 \cup B_3 \big)^c$, and 
let $\F$ be the sigma-algebra generated by $Z^1$, $Z^2$ and $W^1$. Then
$$\E\big[ M_K(X) \1_Q \,\big|\, \F \big] \,=\, \Pr_p\big( Q \,\big|\, \F \big) \E\big[ M_K(X) \,\big|\, \F, Q \big] \,=\, 0,$$
by symmetry, since $T_{|K\setminus V|}$ is equally likely to be positive or negative, and $Q$ implies $|T_{|K\setminus V|}| > |S_{|V \cap K|}|$. Thus by the claim, and since $\F$ determines $f(Z)$, we have
\begin{eqnarray*}
&& \left| \E \big[ f(Z) M_K(X) \big] \right| \; \le \; \left| \E \big[ f(Z) M_K(X) \1_Q \big] \right| \,+\, \P(Q^c) \\[+0.5ex]
&& \hspace{2.13cm} = \; \left| \E \Big[ f(Z) \Ex \big[ M_K(X) \1_Q \,\big|\, \F \big] \Big] \right| \,+\, \P(Q^c) \, = \, O\left( \delta_K(\A,p)^{1/3} \log n \right),
\end{eqnarray*}
 as required.

Thus, it only remains to prove the claim, which follows easily using Chernoff's inequality. We have already shown that $\P(B_1) \le \delta_K(\A,p)^{1/3}$, and so it will suffice to prove corresponding bounds for $B_2$ and $B_3 \cap B_1^c$. The bound for $B_2$ follows using Chernoff and the union bound. Indeed, let $t = |K| \delta_K(\A,p)^{2/3}$, and recall that $X^1$ was chosen uniformly. Thus, by Chernoff's inequality,
\begin{eqnarray*}
\P(B_2) & \le & \sum_{k=1}^{t} \Pr\Big( \big| 2 \cdot \xi_{k,1/2} - k \big| > \sqrt{|K|} \delta_K(\A,p)^{1/3} \log n \Big) \\
& \le & 2\sum_{k=1}^t \exp\left( -\frac{ |K| \delta_K(\A,p)^{2/3} \log^2 n }{ 8k } \right) \; \le \; 2t \cdot e^{-\log^2n/8} \; = \; O\big(\delta_K(\A,p)^{1/3}\big),
\end{eqnarray*}
as required.

Finally, we shall bound the probability of $B_3 \cap B_1^c$; that is, the probability that
$$|V \cap K| \le t = |K| \delta_K(\A,p)^{2/3} \quad \textup{and} \quad |T_{|K\setminus V|}| \le \sqrt{|K|} \delta_K(\A,p)^{1/3} \log n.$$
By Observation~\ref{binmax} and the union bound, we have
$$\Pr\Big( \big| 2 \cdot \xi_{m,1/2} - m \big| \le \sqrt{|K|} \delta_K(\A,p)^{1/3} \log n \Big) \; \le \; \sqrt{|K|} \delta_K(\A,p)^{1/3} \log n \cdot \frac{C_1}{\sqrt{m}},$$
for some constant $C_1 > 0$ and every $m \ge 1$. Since $V$ is determined by the information in $\mathcal{F}$, and since $X^3$ is uniformly distributed, we have
\begin{equation*}
\begin{aligned}
\Pr(B_3\cap B_1^c)\;&=\;\E\big[\Pr(B_1^c\cap B_3|\mathcal{F})\big]\;=\;\E\big[\mathbf{1}_{B_1^c}\cdot\Pr(B_3|\mathcal{F})\big]\\
&\le\; \E\left[\mathbf{1}_{B_1^c}\cdot\sqrt{|K|}\delta_K(\A,p)^{1/3}\log n\frac{C_1}{ \sqrt{ |K \setminus V| } }\right]\\
&\le\;\sqrt{|K|}\delta_K(\A,p)^{1/3}\log n\frac{2C_1}{\sqrt{3|K|}} \; = \; O\left( \delta_K(\A,p)^{1/3} \log n \right),
\end{aligned}
\end{equation*}
where we in the second inequality used that on $B_1^c$, we have $|K\setminus V|\ge |K|-t\ge 3|K|/4$, assuming that $t \le |K|/4$ (since otherwise $\delta_K(\A,p) \ge 1/8$, and the proposition is trivial).
This completes the proof of the claim, and hence of the proposition as well.
\end{proof}

It is now easy to deduce Theorem~\ref{algthm}. We shall use the following straightforward optimization lemma.

\begin{lemma}\label{lma:opt}
If $a_1\ge a_2 \ge \ldots \ge a_n > 0$, then
$$
\max\left\{\sum_{i=1}^n c_i^2 \,:\, c_1\ge c_2\ge\ldots\ge c_n\ge0 \textup{ and } \sum_{i=1}^k c_i
\le\sum_{i=1}^k a_i \textup{ for every } k \in [n] \right\} = \sum_{i=1}^n a_i^2.
$$
\end{lemma}

We are ready to prove Theorem~\ref{algthm}.

\begin{proof}[Proof of Theorem \ref{algthm}]
Let $r \in \N$ be fixed, and let $(f_n)_{n \ge 1}$ be a sequence of monotone functions $f_n \colon \{0,1\}^n \to [0,1]$. For each $n \in \N$, let $\A_1,\ldots,\A_r \in \A^*(f)$
and let $K_1, \ldots, K_r$ be a partition of $[n]$. Let $p \in (0,1)$, and suppose that
$$\delta_{K_i}(\A_i,p) \big( \log n \big)^6 \; \to \; 0$$
as $n \to \infty$ for each $i \in [r]$. We shall show that $I\!I_p(f_n) \to 0$ as $n \to \infty$, and hence deduce, by Theorem~\ref{BKSp}, that $(f_n)_{n\ge1}$ is NS$_p$.

Choose $C > 0$ so that Propositions~\ref{prop:InfMKcorr} and~\ref{prop:algMKcorr} both hold for $C$, and assume that $n \in \N$ is sufficiently large so that $\delta_{K_i}(\A_i,p)^{1/3} \log n \le 1/(2C)$ for each $i \in [r]$. To bound $I\!I_p(f_n) = \sum_{j=1}^n \Inf_{p,j}(f_n)^2$ from above, we shall first bound $\sum_{j \in K} \Inf_{p,j}(f_n)$ for every $K \subset [n]$, and then apply Lemma~\ref{lma:opt}. Let us assume for simplicity that $\delta_{K_i}(\A_i,p) \ge 1/n$ for some $i \in [r]$; the other case follows by an almost identical calculation.

\bigskip
\noindent \textbf{Claim:} For every $K \subset [n]$, we have
$$\sum_{j \in K} \Inf_{p,j}(f_n) \; \le \; \frac{C^2 r}{\min\{p,1-p\}} \sqrt{|K|} \max_{i \in [r]} \Big\{ \delta_{K_i}(\A_i,p)^{1/3} \Big\} \big( \log n \big)^{3/2}.$$

\begin{proof}[Proof of claim]
By Proposition~\ref{prop:InfMKcorr}, for every $K \subset [n]$ we have
$$\sum_{j \in K} \Inf_{p,j}(f_n) \; \le \; \frac{C}{\min\{p,1-p\}} \sqrt{|K|} \E\big[ f_n(Z)M_{K}(X) \big] \left( 1 + \sqrt{- \log\E\big[ f_n(Z) M_{K}(X) \big]} \right).$$
Moreover, by Proposition~\ref{prop:algMKcorr}, for every $i \in [r]$ and every $K \subset K_i$,
$$\E\big[ f_n(Z) M_{K}(X) \big] \,\le\, C\, \delta_{K_i}(\A_i,p)^{1/3} \log n.$$
Note that $x \big(1 + \sqrt{ - \log x} \big)$ is increasing on $(0,1/2)$, and recall that $C\, \delta_{K_i}(\A_i,p)^{1/3} \log n \le 1/2$. Thus, if $K \subset K_i$ for some $i \in [r]$, then
$$\sum_{j \in K} \Inf_{p,j}(f_n) \; \le \; \frac{C^2}{\min\{p,1-p\}} \sqrt{|K|} \max_{i \in [r]} \Big\{ \delta_{K_i}(\A_i,p)^{1/3} \Big\} \big( \log n \big)^{3/2},$$
since $\max_{i \in [r]} \delta_{K_i}(\A_i,p) \ge 1/n$. Summing over $i \in [r]$, the claim follows.
\end{proof}

Without loss of generality, assume that
$$\Inf_{p,1}(f_n) \; \ge \; \ldots \; \ge \; \Inf_{p,n}(f_n),$$
and apply Lemma~\ref{lma:opt} with $c_j = \Inf_{p,j}(f_n)$, and
$$a_j \, = \,  \frac{C^2 r}{\min\{p,1-p\}} \max_{i \in [r]} \Big\{ \delta_{K_i}(\A_i,p)^{1/3} \Big\} \big( \log n \big)^{3/2} \big( \sqrt{j} - \sqrt{j-1} \big).$$
By the claim applied to $K = [k]$, we have, for each $k \in [n]$,
$$\sum_{j=1}^k c_j \; = \; \sum_{j = 1}^k \Inf_{p,j}(f_n) \; \le \; \frac{C^2 r}{\min\{p,1-p\}} \sqrt{k} \max_{i \in [r]} \Big\{ \delta_{K_i}(\A_i,p)^{1/3} \Big\} \big( \log n \big)^{3/2} \;=\; \sum_{j=1}^k a_j,$$
and hence, writing $C^\prime=\big(C^2r/\min\{p,1-p\}\big)^2$, since $p$ is fixed and $\sum_j \big( \sqrt{j} - \sqrt{j-1} \big)^2 = O\big( \log n \big)$, by Lemma~\ref{lma:opt} we have
\begin{eqnarray*}
\sum_{j=1}^n \Inf_{p,j}^2(f_n) & \le & \sum_{j=1}^n a_j^2 \; = \; C^\prime \max_{i \in [r]} \Big\{ \delta_{K_i}(\A_i,p)^{2/3} \Big\} \big( \log n \big)^3 \sum_{j=1}^n \big( \sqrt{j} - \sqrt{j-1} \big)^2 \\
& = & C^\prime \max_{i \in [r]} \Big\{ \delta_{K_i}(\A_i,p)^{2/3} \Big\} \big( \log n \big)^4 \; \to \; 0
\end{eqnarray*}
as $n \to \infty$, as claimed. Thus, by Theorem~\ref{BKSp}, $(f_n)_{n\ge1}$ is NS$_p$, as required.
\end{proof}

We finish this section by proving the following closely related result, which was also proved in~\cite[Theorem~1.6]{BKS} in the case $p = 1/2$ (and for functions into $[0,1]$). In fact we shall not need it, but since it follows immediately from the uniform case and Proposition~\ref{prop:hf}, and may be of independent interest, we include it for completeness.

Given a function $h \colon \{0,1\}^n \to [0,1]$, define
$$\Lambda(h) \; := \; \max_{K \subset [n]} \E\big[ h(X)M_K(X) \big].$$
In particular, $\Lambda(h_f) = \ds\max_{K\subset[n]} \E\big[f(Z)M_K(X)\big]$.

\begin{thm}\label{MKcorrelationnew}
There exists a constant $C > 0$ such that, if $f \colon \{0,1\}^n \to [0,1]$ is monotone and $p \in (0,1)$, then
$$I\!I_p(f) \; \le \; \frac{C}{\min\{p^2,(1-p)^2\}}\Lambda^2(h_f)\big(1 - \log\Lambda(h_f)\big)\log n.$$
\end{thm}

\begin{proof}
We apply the uniform case to the function $h_f$. By Proposition \ref{prop:hf}, it follows that
$$4\bar{p}^2I\!I_p(f) \, = \, I\!I_{1/2}(h_f) \, \le \, C\Lambda^2(h_f) \big( 1 - \log\Lambda(h_f) \big)\log n,$$
as required.
\end{proof}

\section{Hypergraphs}\label{hypersec}

In this section we shall prove Theorem~\ref{varB}, which will allow us to bound the variance (in $\PB_{\lambda_c/p}$) of the probability of crossing a rectangle in the model $\PP^B_p$. Although one can think of all the results in this section in terms of events on the cube $\{0,1\}^n$, it will be convenient for us to use the language of hypergraphs. For background on Graph Theory, see~\cite{MGT}.

Recall that a hypergraph $\HH$ is just a collection of subsets of $[n]$, which we refer to as the edges of $\HH$. We shall write $\HH_m$ for the $m$-uniform hypergraph contained in $\HH$, that is, the collection of edges with $m$ elements, and 
recall that
$$r_{\HH}(B,p) \; := \; \Pr_p^B\big( A \in \HH \big),$$
 where $A$ is a $p$-subset of $B$. Throughout this section, $B$ will denote a $q$-subset of $[n]$.

The proof of Theorem~\ref{varB} is in two parts: first we shall prove the corresponding result for sets $A$ and $B$ of fixed size; then we shall deduce the result for $p$- and $q$-subsets.

\subsection{The proof for sets of fixed size}

Let us begin by informally illustrating the central idea with a simple example.  Let $G$ be a (large) graph with vertex set $[n]$, and consider the restriction of $G$ to a random subset $S \subset [n]$ selected uniformly at random from the sets of size $k$. If $k = 2$ then the resulting graph $G[S]$ will have density either $0$ or $1$, which will typically be quite far from the density of the original graph.  However, once $k$ is a large constant the density of $G[S]$ is already unlikely to be far from the density of $G$. Indeed, it is elementary to bound the variance of this density.

The following proposition extends this result to hypergraphs.   Given a hypergraph $\HH$ on vertex set $[n]$, a subset $S\subset [n]$ and an integer $0 \le m \le n$, define
$$X_{m}(S) \; := \; \big| \big\{ e \in \HH_m \,:\, e \subset S \big\} \big|,$$
and $\tilde{X}_{m}(S) := X_m(S) / {{|S|} \choose m}$.

\begin{prop}\label{var}
Let $n,m,k \in \N$, and suppose that $n \ge k \ge m$, and that $n \ge 3m^3$ and $n\ge km/2$. Let $\HH$ be a hypergraph on vertex set $[n]$, and let $B_k \subset [n]$ be a uniformly chosen subset of size $k$. Then
$$\Var\big( X_m(B_k) \big) \; \le \; \frac{48 m}{k} {k \choose m}^{2}$$
and
$$\Var\big( \tilde{X}_m(B_k) \big) \; \le \; \ds\frac{48 m}{k}.$$
\end{prop}

We remark that with a little extra effort, one could improve the upper bounds in Proposition~\ref{var} by a factor of $\beta_m = e(\HH_m) / {n \choose m}$. Since we shall not need such a strengthening, however, we leave the details to the reader. In order to keep the presentation simple, we also make no attempt to optimize the constant.

We shall use some straightforward relations between binomial coefficients in the proof of Proposition~\ref{var}; we state them here for convenience.

\begin{obs}\label{binobs}
Let $n,k,m,t$ be integers such that $k \ge m\ge t \ge 1$ and $n \ge 3m^3$. Then
\begin{itemize}
\item[$(a)$] $\ds{k \choose m}^2 \, = \, \sum_{t=0}^m {k \choose {2m-t}} {{2m-t} \choose m} {m \choose t}$.\\[+1ex]
\item[$(b)$] $\ds{{m-1} \choose t} {n \choose m}^2 \, \le \, {n \choose {2m-t}} {{2m-t} \choose m} {{n-1} \choose t}$.\\[+1ex]
\item[$(c)$] $\ds{{m-1} \choose {t-1}} {{n-t-1} \choose {m-t}} {n \choose m} \, \le \, \frac{2t}{m} {n \choose {2m-t}} {{2m-t} \choose m} {m \choose t}$.
\end{itemize}
\end{obs}

\begin{proof}
For $(a)$, note that both sides count the number of pairs of $m$-subsets of a fixed $k$-set; on the right-hand side we have partitioned according to their intersection. For $(b)$ and $(c)$ simply cancel common terms, and note that, for fixed $m$ and $t$,
$$\frac{(n-m)!^2}{(n - 2m + t)! (n-t)!} \; \to \; 1$$
as $n \to \infty$. More careful calculation shows that $n \ge 3m^3$ suffices.
\end{proof}


We shall use Bey's inequality in order to prove the following lemma, from which Proposition~\ref{var} follows easily. Let
$$Y_{t}(k,m) \, := \, {k \choose {2m-t}} {{2m-t} \choose m} {m \choose t}$$
denote the number of pairs of $m$-subsets of a fixed $k$-set which have $t$ common elements.

\begin{lemma}\label{lem:var}
Let $k,m,n \in \N$, with $n \ge k \ge 2m$ and $n \ge 3m^3$. Let $\HH$ be a hypergraph on $[n]$, and let $B_k \subset [n]$ be a uniformly chosen subset of size $k$. Then
$$\Var\big( X_{m}(B_k) \big) \; \le \; 2\beta_m \ds\sum_{t=1}^{m} \left(\frac{t}{m}+\frac{k}{2n}\right) Y_t(k,m)
.$$
\end{lemma}

\begin{proof}
Let $\alpha(\HH,t) := \big| \big\{ (e,f) : e,f \in \HH \textup{ and } |e \cap f| = t \big\} \big|$ denote the number of pairs $(e,f)$ of edges of $\HH$ such that $|e \cap f| = t$. We first claim that
\begin{equation}\label{eq:X^2}
\ds\Ex\big[ X_{m}(B_k)^2 \big] \,=\,  \sum_{t=0}^m \alpha(\HH_m,t) \frac{ {k \choose {2m-t}} }{ {n \choose {2m-t}} }.
\end{equation}
Indeed, writing $\1_A$ for the indicator function of the event $A$, and ${[n] \choose k}$ for the collection of subsets of $[n]$ of size $k$, we obtain
$$\Ex\big[ X_{m}(B_k)^2 \big] \; = \; \frac{1}{ {n \choose k} }\sum_{S \in {{[n]} \choose k}} \sum_{e,f \in \HH_m} \1_{\{e \cup f \subset S\}} \; = \; \frac{1}{ {n \choose k} } \sum_{t=0}^m \sum_{e,f \in \HH_m} \1_{\{|e \cap f| = t\}} \sum_{S \in {{[n]} \choose k}}\1_{\{e \cup f \subset S\}}.$$
But if $|e \cap f| = t$ then $|e \cup f| = 2m - t$, and so there are exactly ${{n-2m+t} \choose {k-2m+t}}$ sets $S$ of size $k$ such that $e \cup f \subset S$. Moreover, ${{n-2m+t} \choose {k-2m+t}} {n \choose {2m-t}} = {n \choose k}{k \choose {2m-t}}$, and hence
$$\Ex\big[ X_{m}(B_k)^2 \big] \; = \; \sum_{t=0}^m \sum_{e,f \in \HH_m} \1_{\{|e\cap f| = t\}} \frac{ {k \choose {2m-t}} }{ {n \choose {2m-t}} },$$
as claimed.

Next, observe that $\alpha(\HH_m,t) \le d_2(\HH_m,t)$, where $d_2(\HH_m,t)$ denotes the sum of $d_\HH(T)^2$ over all $t$-sets in $[n]$, and recall that $e(\HH_m) = \beta_m {n \choose m}$. Hence, by Bey's inequality and Observation~\ref{binobs},
\begin{eqnarray}\label{eq:alpha}
\alpha(\HH_m,t)  \frac{ {k \choose {2m-t}} }{ {n \choose {2m-t}} } & \le & \left( \frac{ {m \choose t} { {m-1} \choose t} }{ {{n-1} \choose t} } e(\HH_m)^2 \,+\, {{m-1} \choose {t-1}} {{n-t-1} \choose {m-t}} e(\HH_m) \right)  \frac{ {k \choose {2m-t}} }{ {n \choose {2m-t}} } \nonumber\\
& \le & \left( \beta_m^2 + \frac{2t}{m} \cdot \beta_m \right) {k \choose {2m-t}}{{2m-t} \choose m} {m \choose t}
\end{eqnarray}
for every $1 \le t \le m$. Moreover, $\alpha(\HH_m,0)\le e(\HH_m)^2=\beta_m^2{n \choose m}^2$, so by Observation~\ref{binobs}$(a)$
\begin{equation*}
\alpha(\HH_m,0)  \frac{ {k \choose {2m}} }{ {n \choose {2m}} } \; \le \; \beta_m^2{n \choose m}^2\frac{ {k \choose {2m}} }{ {n \choose {2m}} } \; = \; \beta_m^2\frac{ {k \choose {2m}} }{ {n \choose {2m}} }\sum_{t=0}^m {n \choose {2m-t}} {{2m-t} \choose m} {m \choose t}.
\end{equation*}
Cancelling common terms, we easily see that for each $1\le t\le m$
$$
\frac{{k \choose {2m}}{n \choose {2m-t}} }{{n \choose {2m}}{k\choose 2m-t} } \;\le\; \left(\frac{k}{n}\right)^t \;\le\; \frac{k}{n}.
$$
Hence,
\begin{equation}\label{eq:t=0}
\alpha(\HH_m,0)  \frac{ {k \choose {2m}} }{ {n \choose {2m}} } \; \le \; \beta_m^2{k \choose {2m}} {{2m} \choose m}+\frac{k}{n}\cdot\beta_m^2\sum_{t=1}^m {k \choose {2m-t}} {{2m-t} \choose m} {m \choose t}.
\end{equation}
Finally,
\begin{equation}\label{eq:Ex}
\Ex\big[ X_{m}(B_k) \big]^2 \; = \; \beta_m^2 {k \choose m}^2 \; = \; \beta_m^2\sum_{t=0}^m {k \choose {2m-t}} {{2m-t} \choose m} {m \choose t},
\end{equation}
by Observation~\ref{binobs}$(a)$. Combining~\eqref{eq:X^2},~\eqref{eq:alpha},~\eqref{eq:t=0} and~\eqref{eq:Ex}, we obtain
\begin{eqnarray*}
\Var\big( X_{m}(B_k) \big) 
& \le & \sum_{t=1}^m \left(\frac{2t}{m}\cdot\beta_m+\frac{k}{n}\cdot\beta_m^2 \right) {k \choose {2m-t}} {{2m-t} \choose m} {m \choose t},
\end{eqnarray*}
as required.
\end{proof}

It is easy to deduce Proposition~\ref{var} from Lemma~\ref{lem:var}.

\begin{proof}[Proof of Proposition~\ref{var}]
We shall prove the claimed bound on $\Var\big( X_{m}(B_k) \big)$; the second statement follows immediately from the first, since $X_m(B_k) = {k \choose m} \tilde{X}_{m}(B_k)$. The result is trivial for $m \le k \le 48m$, so we can assume that $k \ge 48m$. (In fact we shall only use that $k \ge 4m$.)

First, note that by Lemma~\ref{lem:var}, and since $n \ge km/2$, we have
\begin{equation}\label{ytbound}
\Var\big(X_{m}(B_k)\big) \;\le\; 2 \sum_{t=1}^{m} \frac{t+1}{m}Y_{t}(k,m),
\end{equation}
where $Y_{t}(k,m) = {k \choose {2m-t}} {{2m-t} \choose m} {m \choose t}$. We shall see that most of the weight of the $Y_{t}(k,m)$ is concentrated on terms with small $t$. We split into two cases, depending on the size of $m$.

\bigskip
\noindent \textbf{Case 1:} $k \ge 3m^2$.

\bigskip
We shall prove that
\begin{equation}\label{eq:C1}
\sum_{t=1}^m\frac{t+1}{m} Y_t(k,m)\;\le\; \frac{4}{m} Y_1(k,m) \; \le \;  \frac{4m}{k} {k\choose m}^2.
\end{equation}
Indeed, first note that
\begin{equation}\label{eq:expdec}
\ds\frac{(t+2)Y_{t+1}(k,m)}{(t+1)Y_{t}(k,m)} \; = \; \frac{(t+2)(m-t)^2}{(t+1)^2(k-2m+t+1)} \; \le \; \frac{3m^2}{2(t+1)(k-2m)} \; \le \;  \frac{1}{2},
\end{equation}
since $k - 2m \ge k/2$ and $(t+1)k \ge 2k \ge 6 m^2$. This proves the first inequality in~\eqref{eq:C1}; for the second, observe that
$$Y_1(k,m) \;=\; \frac{k(k-1)\ldots(k-2m+2)}{(m-1)!^2} \; \le \; \frac{m^2}{k} {k \choose m}^2,$$
as claimed. By~\eqref{ytbound}, we obtain $\Var\big(X_{m}(B_k)\big) \le \ds\frac{8m}{k} {k\choose m}^2$.

\bigskip
\noindent \textbf{Case 2:} $k \le 3m^2$.

\bigskip
Let $a := \big\lfloor 6 m^2 / k \big\rfloor$, and observe that~\eqref{eq:expdec} holds whenever $t \ge a$. Thus
\begin{equation}
\sum_{t=a}^m \frac{t+1}{m} Y_t(k,m) \; \le \; \frac{2(a+1)}{m} Y_{a}(k,m) \; \le \;  \frac{18m}{k} {k\choose m}^2,
\end{equation}
since $Y_{a}(k,m) \le {k\choose m}^2$ and $a+1 \le 9m^2/k$. Moreover, it is immediate that
\begin{equation}
\sum_{t=1}^{a-1} \frac{t+1}{m} Y_t(k,m) \; \le \; \frac{a}{m} \sum_{t=1}^{a-1} Y_t(k,m) \; \le \;  \frac{6m}{k} {k\choose m}^2.
\end{equation}
By~\eqref{ytbound}, we obtain $\Var\big(X_{m}(B_k)\big) \le \ds\frac{48m}{k} {k\choose m}^2$, as required.
\end{proof}

\subsection{The proof for random-sized sets}

We shall now deduce Theorem~\ref{varB} from Proposition~\ref{var}. Using the conditional variance formula $\Var(X)=\Var\big(\Ex[X|\,Y]\big)+\Ex\big[\Var(X|\,Y)\big]$, the variance we want to control may be expressed as
\begin{equation}\label{eq:condvar}
\Var_q\big(r_\HH(B,p)\big)\,=\,\Var_q\big(\Ex_q\big[r_\HH(B,p)\big|\,|B|\big]\big)+\Ex_q\big[\Var_q\big(r_\HH(B,p)\big|\,|B|\big)\big],
\end{equation}
where $|B|$ denotes the size of the set $B$. The latter of the two terms can be controlled using Proposition~\ref{var} and Chernoff's inequality. The challenge will be the former term. To see, heuristically, why $\Var_q\big(\Ex_q\big[r_\HH(B,p)\big|\,|B|\big]\big)$ should be small, note that $|B|$ will roughly fluctuate by $\sqrt{qn}$ around its mean. This will influence the size of a $p$-subset $A$ of $B$ roughly by $p\sqrt{qn}$. However, $|A|$ will naturally vary by $\sqrt{pqn}$ which is much larger than $p\sqrt{qn}$ when $p$ is small. Hence, conditioning on the size of $B$ will not affect the size of $A$ much, and should imply that the former term in~\eqref{eq:condvar} is small.


We begin with the latter term in~\eqref{eq:condvar}. The first step is to prove the result corresponding to Proposition~\ref{var} for fixed size $k$ of $B$ and a randomly chosen subset thereof. Indeed, given a hypergraph $\HH$ on vertex set $[n]$, a subset $S \subset [n]$ of size $k$, and $p \in (0,1)$, observe that
$$r_\HH(S,p) \, = \, \sum_{m=0}^k \Pr\big( \xi_{k,p} = m \big) \tilde{X}_m(S),$$
where $\xi_{k,p} \sim \Bin(k,p)$, as in the previous section. The following proposition is an easy consequence of Proposition~\ref{var}.

\begin{prop}\label{prop2}
Let $p \in (0,1)$ and let $n,k \in \N$, with $n \ge 24(pk)^3$ and $n\ge 2pk^2$. Let $\HH$ be a hypergraph on vertex set $[n]$, and let $B_k \subset [n]$ be a uniformly chosen subset of size $k$. Then
$$\Var\big( r_\HH(B_k, p) \big) \, \le \, 96p \,+\, 4\exp\big( -pk/16 \big).$$
\end{prop}

\begin{proof}
The result follows from Proposition~\ref{var} and Chernoff's inequality, since if $m \le 2pk$ then the variance of $\tilde{X}_m(B_k)$ is at most $96p$, and the probability that $\xi_{k,p} > 2pk$ is at most $2e^{-pk/16}$.

To spell it out, note that if $p \le 1/2$ and $m \le 2pk$, then $m\le k$, $n \ge 3m^3$ and $n\ge km/2$, and so, by Proposition~\ref{var},
$$
\Var\big( \tilde{X}_{m}(B_k)\big) \,\le\, \frac{48m}{k} \, \le \, 96p.
$$
Since $\Var\big(\tilde{X}_m(B_k)\big) \le 1$, the same bound trivially holds for $ p > 1/2$.

Now since $ r_\HH(B_k, p) = \sum_{m=0}^k \Pr\big( \xi_{k,p} = m \big) \tilde{X}_m(B_k)$,
$$
\Var\big( r_\HH(B_k, p) \big) \; = \; \sum_{m_1,m_2} \Pr\big( \xi_{k,p} = m_1 \big) \Pr\big( \xi_{k,p} = m_2 \big) \Cov\big(\tilde{X}_{m_1}(B_k),\tilde{X}_{m_2}(B_k)\big).
$$
By Cauchy-Schwarz, we have $\Cov(X,Y) \le \sqrt{\Var(X)\Var(Y)}$, and thus
\begin{eqnarray*}
\Var\big( r_\HH(B_k, p)\big) & \le & \sum_{m_1,m_2} \Pr\big( \xi_{k,p} = m_1 \big) \Pr\big( \xi_{k,p} = m_2 \big) \sqrt{\Var\big(\tilde{X}_{m_1}(B_k)\big)\Var\big(\tilde{X}_{m_2}(B_k)\big)} \\
& \le & 96p \,+ \, 2 \cdot \Pr\big( \xi_{k,p} > 2pk \big) \; \le \; 96p \,+\, 4 \exp\big( -pk/16 \big),
\end{eqnarray*}
where the last step is by Chernoff's inequality, as required.
\end{proof}

In order to deduce Theorem~\ref{varB} from Proposition~\ref{prop2}, we shall use the following simple bounds on binomial random variables, which follow immediately from Chernoff's inequality.

\begin{obs}\label{bincalc}
Let $p\in(0,1/2]$, $q\in(0,1)$ and $n \in \N$, with 
$pqn \ge 32\log(1/p)$.
Then
\begin{itemize}
\item[$(a)$] $\Pr\Big( \big| \xi_{n,q}  - qn \big| \, > \, 2 \sqrt{qn \log(1/p) } \Big) \, \le \, 2p$.\\[-0ex]
\item[$(b)$] $\Pr\left( \xi_{n,q}  < \ds\frac{16}{p} \log \frac{1}{p}  \right) \, \le \, \Pr\Big( \xi_{n,q} < qn/2\Big) \, \le \, 2e^{-qn/16} \, \le \, 2p$.
\end{itemize}
\end{obs}

We shall also need the following bound, relating nearby binomial coefficients. 

\begin{lemma}\label{Pcalc}
Let $p\in(0,1/4]$, $q\in(0,1)$ and $n \in \N$ be such that $pqn \ge 32\log(1/p)$.
If
$$qn \,-\, 2\sqrt{qn \log(1/p) } \; \le \; k \;\le\; k' \;\le\; qn \,+\, 2\sqrt{qn \log(1/p) },$$
and $\big| m - pqn \big| \le 4 \sqrt{pqn \log(1/p)}$, then
\[
\frac{\Pr\big( \xi_{k',p} = m \big) }{ \Pr\big( \xi_{k,p} = m \big)}=1 + O\bigg( \sqrt{p}\log \frac{1}{p} \bigg).
\]
\end{lemma}

Since this lemma follows from a straightforward calculation, we shall only sketch the proof of the upper bound.

\begin{proof}[Sketch of proof]
We claim that, for each $k \le \ell \le k'$,
$$\frac{\Pr\big( \xi_{\ell,p} = m \big)}{\Pr\big( \xi_{\ell-1,p} = m \big)} \; = \; \frac{(1-p)\ell}{\ell - m} \; \le \; 1 + O\left( \frac{\sqrt{p \log(1/p)}}{\sqrt{qn}} \right).$$
Indeed, note that $m \le p\ell + 6\sqrt{pqn \log \frac{1}{p}}$, and so
$$\frac{(1-p)\ell}{\ell - m} \; \le \; \left( 1 - \frac{8\sqrt{pqn \log(1/p)}}{\ell} \right)^{-1} \; \le \; 1 + O\left( \frac{\sqrt{p \log(1/p)}}{\sqrt{qn}} \right).$$
Now we simply take the product over $\ell$ and note that this range is at most of size
$4\sqrt{qn \log(1/p) }$, to obtain
$$\frac{\Pr\big( \xi_{k',p} = m \big)}{\Pr\big( \xi_{k,p} = m \big)} \; \le \; \exp\bigg[ O\bigg( \sqrt{p} \log \frac{1}{p} \bigg) \bigg] \; = \; 1 + O\bigg( \sqrt{p} \log \frac{1}{p} \bigg),$$
as required. The proof of the lower bound is the same.
\end{proof}

We are now ready to prove Theorem~\ref{varB}.

\begin{proof}[Proof of Theorem~\ref{varB}]
Note that the result is trivial for any $p\in(\frac{1}{4},\frac{1}{2}]$, since the variance is at most 1. Let $0 < p\le 1/4$, $0<q < 1$ and $n \in \N$, and suppose that $n \ge 200(pqn)^3$, $n\ge 8p(qn)^2$ and $pqn \ge 32\log(1/p)$.  Observe that therefore Observation~\ref{bincalc} and Lemma~\ref{Pcalc} hold, and that Proposition~\ref{prop2} holds for every $k \le 2qn$.

For each $k \in [n]$, let $\alpha_k := \Ex\big[ r_\HH(B_k,p) \big]$, where $B_k \subset [n]$ is a uniformly chosen set of size~$k$. Let $K$ and $K'$ be independent random variables, each with distribution $\Bin(n,q)$.

\medskip
\noindent \textbf{Claim 1}: $\Var_q\big(  r_\HH(B,p) \big) \, = \, \ds\frac{1}{2} \Ex\Big[ \big(\alpha_{K}-\alpha_{K'}\big)^2 \Big] \,+\, O(p)$.

\begin{proof}[Proof of Claim 1]
Applying the conditional variance formula we obtain as in~\eqref{eq:condvar}
$$
\Var_q\big(  r_\HH(B,p) \big) \, = \, \Var\Big( \Ex\big[ r_\HH(B,p) \,\big|\, |B| \big] \Big) \,+\, \sum_{k=0}^n \Var\Big( r_\HH(B,p) \,\big|\, |B| = k \Big) \Pr\big( \xi_{n,q} = k \big).
$$
Note that in the first term, the variance is over the choice of $\ell := |B|$, and the expectation over the (uniform) choice of a set $B$ of size $\ell$. In the sum, the variance is over the uniform choice of a $k$-set $B$.

Now, $\Ex\big[ r_\HH(B,p) \,\big|\, |B| = k \big] = \alpha_k$, so the first term may be re-written as
$$\Var\big( \alpha_K \big) \; = \; \ds\frac{1}{2} \Ex\Big[ \big(\alpha_{K}-\alpha_{K'}\big)^2 \Big].$$
For the sum, first note that $\Pr\big( \xi_{n,q} < \frac{16}{p} \log \frac{1}{p} \big) \le 2p$, by Observation~\ref{bincalc}. On the other hand, $\Pr(\xi_{n,q}> 2qn)\le 2\exp(-qn/16)\le 2p$, by Chernoff. For $k\le 2qn$ we have $24(pk)^3\le 192(pqn)^3\le n$ and $2pk^2\le 8p(qn)^2\le n$, so if $\frac{16}{p} \log \frac{1}{p}\le k \le 2qn$, then, by Proposition~\ref{prop2}, we have
$$
\Var\Big( r_\HH(B_k,p) \Big) \,=\, O(p),
$$
and the claim follows.
\end{proof}

By Observation~\ref{bincalc}, the probability that
\begin{equation}\label{KK'}
qn - 2\sqrt{qn \log\frac{1}{p} } \, \le \, K,K' \, \le \, qn + 2 \sqrt{qn \log\frac{1}{p} }
\end{equation}
is at least $1 - 4p$. Hence, it will suffice to prove the following claim.

\bigskip
\noindent \textbf{Claim 2}: If $qn - 2\sqrt{qn \log\frac{1}{p} } \le k \le k' \le qn + 2\sqrt{qn \log\frac{1}{p} }$, then $$\big| \alpha_{k'} - \alpha_k \big| \, = \, O\bigg( \sqrt{p} \log \frac{1}{p} \bigg).$$

\begin{proof}[Proof of Claim 2]
Note that
$$
\alpha_k \; = \; \Ex\big[ r_\HH(B_k,p)\big] \; = \; \sum_{m=0}^k \Pr\big( \xi_{k,p} = m \big) \Ex\big[ \tilde{X}_m(B_k) \big] \; = \; \sum_{m=0}^k \Pr\big( \xi_{k,p} = m \big) \beta_m,
$$
and set $S = \big\{m : |m - pqn| \le 4\sqrt{pqn \log(1/p)} \big\}$. We assume for simplicity that $\alpha_{k'} \ge \alpha_k$; the other case is the same. We have
\begin{eqnarray*}
\alpha_{k'} -\alpha_k & = & \sum_{m=0}^{k'} \Pr\big( \xi_{k',p} = m \big) \beta_m - \sum_{m=0}^k \Pr\big(\xi_{k,p} = m\big) \beta_m\\
& \le & \sum_{m \in S} \left| \Pr\big( \xi_{k',p} = m \big) - \Pr\big(\xi_{k,p} = m\big) \right| \,+\, O(p),
\end{eqnarray*}
since, by Chernoff's inequality and the triangle inequality, $\Pr\big( \xi_{k,p} \not\in S \big)+\Pr\big( \xi_{k',p} \not\in S \big) = O(p)$. 
Moreover,  by Lemma~\ref{Pcalc}, we have
$$\big| \Pr\big( \xi_{k',p} = m \big) - \Pr\big(\xi_{k,p} = m\big) \big| \; = \; O\left( \sqrt{p} \log\frac{1}{p} \right) \cdot \Pr\big(\xi_{k,p} = m\big)$$
for every $m \in S$. We conclude that
$$\big| \alpha_{k'} -\alpha_k \big| \; = \; O\left( \sqrt{p} \log\frac{1}{p} \right),$$
as required.
\end{proof}

It is now easy to deduce Theorem~\ref{varB} from Claims~1 and~2, and~\eqref{KK'}. Indeed,
\begin{eqnarray*}
\Var_q\big(  r_\HH(B,p) \big) & = & \ds\frac{1}{2} \Ex\Big[ \big(\alpha_{K}-\alpha_{K'}\big)^2 \Big] \,+\, O(p)\\
& \le & \frac{1}{2}\cdot2 \cdot \Pr\Big( \big| \xi_{n,q} - qn \big| > 2 \sqrt{qn \log(1/p) } \Big) \,+\, O\left( p \bigg( \log \frac{1}{p} \right)^2 \bigg)\\
& = & O\bigg( p \left( \log \frac{1}{p} \right)^2 \bigg),
\end{eqnarray*}
as required.
\end{proof}

We end the section by noting a couple of easy consequences of Theorem~\ref{varB}. The first is a generalization, to an arbitrary event on $\{0,1\}^n$, of the method we shall use to deduce Proposition~\ref{Bprop} from Theorem~\ref{nontriv}. It follows almost immediately from Theorem~\ref{varB}, via Chebychev's inequality.
First, note that a $p$-subset of a $q$-subset of $[n]$ is a $pq$-subset of $[n]$. In particular,
\begin{equation}\label{expect}
\Ex_q \big[ r_\HH(B,p) \big] \, =\,  \Ex_q\big[\Pr_p^B(A\in\HH)\big]\,=\, \Pr_{pq}(A \in \HH).
\end{equation}

\begin{cor}\label{gencor}
For every $\eps > 0$ there exists $p_0=p_0(\eps) > 0$ such that if $p\in(0,p_0)$, $r\in(0,p)$, and $n\in\N$ satisfy $n \ge 200(rn)^3$, $n\ge 8(rn)^2/p$ and $rn \ge 32 \log\frac{1}{p}$, then
\[
\Pr_{r/p} \Big( \Big|\Pr_p^B(A\in\HH) - \Pr_r(A\in\HH)\Big|>\eps\Big) \, < \, \eps,
\]
for every event (or hypergraph) $\HH \subset \{0,1\}^n$.
\end{cor}

\begin{proof}
Given $\eps>0$, choose $p_0>0$ in accordance to Theorem~\ref{varB} such that
$$\Var_{r/p} \big( r_\HH(B,p) \big) \, < \, \eps^3$$
for all $p< p_0$, and for all $\HH$.
By Chebychev's inequality and \eqref{expect},
$$\Pr_{r/p} \Big( \big|r_\HH(B,p)- \Pr_r(A\in\HH)\big| > \eps \Big)  \; \le \; \frac{ \Var_{r/p} \big( r_\HH(B,p) \big) }{ \eps^2 } \; < \; \eps,$$
as required.
\end{proof}

Finally, we state the following extremal result of hypergraphs, which may be of independent interest, and follows easily from Proposition~\ref{var}. Say that a hypergraph $\HH$ is \emph{$\delta$-quasi monotone} with respect to $(k,m)$ if
$$\left| \big\{ e \in \HH_m :  e \subset B \big\} \right| \,\ge\, \big( 1 - \delta \big) {k \choose m}$$
for every $B \in \HH_k$.

\begin{cor}\label{hyper}
For each $\delta > 0$, there exists $C > 0$ such that the following holds for every $k,m \in \N$, and every $n \ge \max\{3m^3, km/2\}$. If $\HH$ is a $\delta$-quasi monotone hypergraph with respect to $(k,m)$, then either
$$|\HH_m| \,\ge\, \big( 1 - 2\delta \big){n \choose m} \qquad \textup{or} \qquad |\HH_k| \, \le \, \frac{Cm}{k} {n \choose k}.$$
\end{cor}

\begin{proof}
Note that $\Ex\big[ X_m(B_k) \big] = \beta_m {k \choose m}$, and apply Chebychev, using Proposition~\ref{var} to bound the variance. The theorem follows with $C = O(1/\delta^2)$.
\end{proof}

\section{Proof of Theorem~\ref{noise}}\label{T1sec}

In this section we shall put together the pieces, and prove Theorem~\ref{noise}. We shall first deduce Propositions~\ref{varBp} and~\ref{Bprop} from Theorem~\ref{varB}; then we shall use the deterministic algorithm method to prove Theorem~\ref{BpNS}; finally we shall deduce Theorem~\ref{noise}.

Throughout this section, $B$ will be chosen according to $\PB_{\lambda_c/p}$, and $\eta$ will be chosen according to the conditional measure $\Pr^B_p$.

\subsection{Variance bound -- Proof of Propositions~\ref{varBp} and~\ref{Bprop}}

We shall prove the following slight generalization of Proposition~\ref{varBp}, which follows easily from Theorem~\ref{varB}, together with an easy discretization argument.

\begin{prop}\label{strongvarBp}
$$\lim_{p \to 0}\, \limsup_{a,b \to \infty} \, \VarB_{\lambda_c/p} \Big( \P^B_p\big(H ( \eta, R_{a \times b}, \bullet)\big) \Big)\, =\, 0.$$
\end{prop}

In order to apply Theorem~\ref{varB} we need to construct a discrete probability space which closely approximates the continuous space with measure $\PB_{\lambda_c/p}$ on the rectangle $R_{a \times b}$. In order to do so, for each $\delta > 0$ consider the lattice
$$\Lambda \,=\, \Lambda_{a,b}^{\delta} \,:=\, R_{(a+2) \times (b+2)} \cap \delta \ZZ^2,$$
and set $n = |\Lambda_{a,b}^{\delta}|$, the number of vertices of $\delta \ZZ^2$ in the rectangle $R_{(a+2) \times (b+2)}$. (Note that we consider the rectangle $R_{(a+2) \times (b+2)}$ because it contains all the points which can affect the event $H(\eta, R_{a \times b}, \bullet)$.) Let $p > 0$, set $q = q(n)$ to be
\begin{equation}\label{qdef}
q \, := \, 1 - e^{-\lambda_c \delta^2 / p},
\end{equation}
and note that $pqn \, \approx \, \lambda_c ab$. In this subsection, since we shall be dealing with both continuous and discrete probability spaces, we shall write $\hat{B} \subset \Lambda$ to denote a $q$-subset of $\Lambda$ and $\hat{\eta} \subset \hat{B}$ to denote a $p$-subset of $\hat{B}$. Recall that $\Pr^\Lambda_q$ and $\Pr^{\hat{B}}_p$ denote the corresponding probability measures.

The following lemma is an immediate consequence of Theorem~\ref{varB}.

\begin{lemma}\label{gridvar}
For every $p\in\big(0,\frac{1}{2}\big]$, if $a=a(p), b=b(p) \ge 1$ are sufficiently large and $\delta = \delta(a,b) > 0$ is sufficiently small, then the following holds. Let $\Lambda = \Lambda_{a,b}^\delta$, $n = |\Lambda|$ and $q > 0$ be as defined in~\eqref{qdef}. Then
$$\Var^\Lambda_q\Big( \P^{\hat{B}}_p\big(H ( \hat{\eta}, R_{a\times b}, \bullet)\big)\Big) \, = \, O\bigg( p \left( \log \frac{1}{p} \right)^2 \bigg).$$
\end{lemma}

\begin{proof}
We apply Theorem~\ref{varB} to the hypergraph $\HH$ which encodes crossings of the rectangle $R_{a \times b}$. That is, we identify the vertices of $\Lambda = \Lambda_{a,b}^{\delta}$ with the elements of $[n]$, and set
$$\hat{\eta} \in \HH \quad \Leftrightarrow \quad H\big( \hat{\eta}, R_{a \times b}, \bullet \big) \textup{ occurs.}$$
It only remains to check that the conditions of Theorem~\ref{varB} are satisfied if $\delta$ is sufficiently small. To see this, simply note that, by our choice of $q$, we have that $pqn = \Theta(ab)$ and $n/ab \to \infty$ as $\delta \to 0$.
\end{proof}

In order to deduce Proposition~\ref{strongvarBp}, we need to provide a coupling between our two probability spaces -- one discrete, the other continuous -- which approximately maps the crossing event $H(\eta,R_{a \times b},\bullet)$ onto $H(\hat{\eta},R_{a \times b},\bullet)$. In fact this is easy: simply map each point of $B$ into the $\ell_\infty$-nearest point of $\Lambda = \Lambda_{a,b}^\delta$, and take $\delta = \delta(a,b)$ sufficiently small.


To spell it out, cover $R_{(a+2) \times (b+2)}$ with disjoint $\delta \times \delta$ squares, centred on elements of $\Lambda^\delta_{a,b}$, and let $\psi$ map points of $B$ to the centre of the square in which they lie. Given a finite subset $B \subset R_{(a+2) \times (b+2)}$, let
$$\hat{B} \; := \; \big\{ y \in \Lambda^\delta_{a,b} \,:\, \psi(x) = y \textup{ for some } x \in B  \big\},$$
and observe that if $B \subset R_{(a+2) \times (b+2)}$ is chosen according to $\PB_{\lambda_c/p}$, then $\hat{B}$ is a $q$-subset of $\Lambda^\delta_{a,b}$, where $q = 1 - e^{-\lambda_c \delta^2 / p}$, as before.

We define a bad event $E^\delta_{a,b}$ on subsets $B \subset \RR^2$, by saying that $E$ occurs if either of the following holds:
\begin{itemize}
\item[$(a)$] Two points of $B$ lie in the same $\delta \times \delta$ square, i.e., $|\psi^{-1}(y) \cap B| > 1$ for some $y \in \Lambda_{a,b}^\delta$.
\item[$(b)$] There exist $x,y \in B$ with $2 - 2\delta \le \| x - y \|_2 \le 2 + 2\delta$.
\item[$(c)$] There exist $x\in B$ such that $1-\delta \le \|x-\partial R_{a \times b}\| \le 1+\delta$,
\end{itemize}
where $\partial R_{a \times b}$ denotes the boundary of $R_{a \times b}$ and the distance between a point
and a set is defined in the canonical way.
Observe that if $E^\delta_{a,b}$ does not occur, then the graphs naturally induced by the points of $B$ and $\hat{B}$ are identical, and the vertices are in 1-1 correspondence. Hence, conditional on $\left\{E^\delta_{a,b}\right\}^c,$ the events $H(\eta,R_{a \times b},\bullet)$ and $H(\hat{\eta},R_{a \times b},\bullet)$ have the same probability, in $\Pr^B_p$ and $\Pr^{\hat{B}}_p$ respectively.

The following lemma shows that, if $\delta$ is sufficiently small, then the coupling $(B,\hat{B})$ has the desired properties.

\begin{lemma}\label{lem:disc}
For every $a,b \ge 1$ and $p > 0$, there is $\delta_0 = \delta_0(a,b,p) > 0$ such that for all $\delta\in(0,\delta_0)$
$$\PB_{\lambda_c/p} \Big( \P^B_p\big(H ( \eta, R_{a \times b}, \bullet) \ne \P^{\hat{B}}_p\big( H ( \hat{\eta}, R_{a \times b}, \bullet ) \big) \Big) \, \le \, \PB_{\lambda_c/p} \big( E^\delta_{a,b} \big) \; \le \; \sqrt{\delta}.$$
\end{lemma}

\begin{proof}
The first inequality holds by the observations above, since if $E^\delta_{a,b}$ does not occur, then the graphs defined by the points of $B$ and $\hat{B}$ are identical.

To bound $\PB_{\lambda_c/p} \big( E^\delta_{a,b} \big)$, we estimate the probabilities of $(a),$ $(b)$ and $(c).$
For property $(a)$, this is $O(\delta^2 a b / p^2)$, since each square has probability $O(\delta^4  / p^2)$ of containing at least two points of $B$, and there are $O(ab/\delta^2)$ such squares.

For property $(b)$, it is $O(\delta ab / p^2)$. Informally, the reason is that two points uniformly
distributed in $R_{a \times b}$ has probability of order $\delta/(ab)$ of falling within the right distance
of each other. Furthermore, since there are order $(\lambda_c ab/p)^2$ pairs of points, we arrive at the
claimed probability. It is not hard to make this argument precise.

For property $(c),$ a similar argument shows that the probability is $O(\delta ab)$.
\end{proof}

We can now easily deduce Proposition~\ref{strongvarBp} from Lemmas~\ref{gridvar} and~\ref{lem:disc}.

\begin{proof}[Proof of Proposition~\ref{strongvarBp}]
Let $p > 0$, and choose $a,b \ge 1$ sufficiently large and $\delta\in(0, p^2)$ small enough for Lemma~\ref{gridvar} to hold. Let $n = |\Lambda_{a,b}^\delta|$ and $q > 0$ be as described above. Then, by Lemmas~\ref{gridvar} and~\ref{lem:disc},
\begin{eqnarray*}
\VarB_{\lambda_c/p} \Big( \P^B_p\big(H ( \eta, R_{a \times b}, \bullet) \big) \Big) & \le & \Var^\Lambda_q\Big( \P^{\hat{B}}_p\big(H ( \hat{\eta}, R_{a\times b}, \bullet)\big)\Big) + \sqrt{\delta}\; = \; O\bigg( p \left( \log \frac{1}{p} \right)^2 \bigg).
\end{eqnarray*}
Since this is $o(1)$ as $p \to 0$, the proposition follows.
\end{proof}

Our second corollary of Theorem~\ref{varB} is Proposition~\ref{Bprop}. To save us from repeating the discretization, we shall deduce it from Proposition~\ref{strongvarBp}.

\begin{proof}[Proof of Proposition~\ref{Bprop}]
Let $t,\gamma > 0$, and recall that, by Theorem~\ref{nontriv}, we have
\begin{equation*}\label{themean}
c \, \le \, \ExB_{\lambda_c/p}\Big[ \Pr^B_p\big( H( \eta, R_{N \times tN}, \bullet) \big)\Big ]\, \le \, 1 - c
\end{equation*}
for some $c(t) > 0$. Moreover, by Proposition~\ref{strongvarBp}, there exists a constant $p_0=p_0(t,\gamma) > 0$ such that
\begin{equation*}\label{thevar}
\limsup_{N \to \infty} \, \VarB_{\lambda_c/p} \Big( \P^B_p\big(H ( \eta, R_{N\times tN}, \bullet)\big) \Big) \,<\, \frac{c^2\gamma }{4},
\end{equation*}
for every $0 < p < p_0$. Now, setting $c' = c/2$, we obtain
$$\PB_{\lambda_c/p}\Big( \Pr^B_p\Big( H\big( \eta, R_{N \times tN}, \bullet \big) \Big) \not\in(c',1-c') \Big) \, < \, \gamma,$$
for every sufficiently large $N \in \N$, by Chebychev's inequality, as required.
\end{proof}

\subsection{$(f^B_N)_{N \ge 1}$ is NS$_p$: Proof of Theorem~\ref{BpNS}}\label{ProofSubsec}

Our proof that the model $\PP^B_p$ is noise sensitive (for $\PB_{\lambda_c/p}$-almost every $B$) for all sufficiently small $p > 0$ is based on Theorem~\ref{algthm}, the deterministic algorithm method. We begin by defining the algorithm which we shall use; it is a straightforward adaptation of that used in~\cite{BKS} to prove noise sensitivity of crossings in bond percolation.

Recall that, given a finite set $B \subset R_{N+2}$, the function $f_{N}^{B} : \{0,1\}^{B}\to \{0,1\}$ is defined by
$$f^B_N(\eta) = 1 \quad \Leftrightarrow \quad H\big( \eta, R_N, \bullet \big) \textup{ occurs.}$$
The following `Water Algorithm' determines $f^B_N(\eta)$ for any finite set $B \subset R_{N+2}$, and any $\eta \in \{0,1\}^B$. The name of the algorithm is inspired by the following way of visualizing it: imagine pouring water into the left-hand side of $R_N$, and allowing it to fill every ball $D(x)$ which it meets such that $\eta(x) = 1$, i.e., for which $x \in \eta$. If water can flow to the other side of $R_N$, then the event $H(\eta,R_N,\bullet)$ holds.

\begin{alg}
Let $B \subset R_{N+2}$ and let $\eta \in \{0,1\}^B$. Let
$$A_0 \,:=\, \big\{ (x,y) \in \RR^2 : x = -N/2 \big\}$$
and $Q_0 := \emptyset$ denote the `active' and `queried' points at time zero. For each $k \in \N$, if $Q_{k-1}$ and $A_{k-1}$ have already been chosen, then define $Q_{k}$ and $A_{k}$ as follows:
\begin{enumerate}
\item[$1.$] Set $Q_k := D\big( D(A_{k-1}) \big) \cap B$, and query the elements of $Q_k$.\\[-2ex]
\item[$2.$] Let $A_k$ denote the set $x \in Q_k$ such that $\eta(x) = 1$. \\[-2ex]
\item[$3.$] If $A_k = A_{k-1}$, then stop, and set $A_\infty = A_k$, otherwise go to step 1. \\[-2ex]
\item[$4.$] If $H ( A_\infty, R_N, \bullet)$ holds, then output 1, otherwise output 0.
\end{enumerate}
Define the $*$-Water Algorithm $\A_W^*$ similarly, except with $A_0 := \big\{ (x,y) \in \RR^2 : x = N/2 \big\}$.
\end{alg}

To see that the Water Algorithm determines $f^B_N(\eta)$, simply note that an element $x \in B$ is queried if and only if there is a path from the left edge of $R_N$ to $D(x)$, using only points of $D(\eta) \cap R_{N+2}$. Thus, if $f_N^B(\eta) = 1$ then the algorithm will find a horizontal path across $R_N$ in $D(\eta) \cap R_N$; conversely, if $f_N^B(\eta) = 0$ then the algorithm will output zero, since $A_\infty \subset \eta$.



\smallskip
We shall apply the Water Algorithm for the vertices in the right-hand half of $R_{N+2}$, and the $*$-Water Algorithm for those in the left-hand half. 
Define
$$K^L_{N} := R_{N+2} \cap \Big( \big( -\infty, 0 \big) \times \RR \Big) \quad \text{and} \quad K^R_{N} := R_{N+2} \cap \Big( \big[0, \infty \big) \times \RR \Big),$$
and recall that for fixed $B \subset \RR^2$ and $v \in B \cap K^R_N$, 
\[
\delta_v(\A_W,p) \,=\, \Pr^B_p\Big( \A_W \textrm{ queries $v$ when determining } f_N^B(\eta)\Big)
\]
and $\delta_{B \cap K^R_N}(\A_W,p) = \ds\max_{v \in B \cap K^R_N} \delta_v(\A_W,p)$. 

The following lemma will allow us to apply Theorem~\ref{algthm}. 

\begin{lemma}\label{reveal:A}
For every $C > 0$, there exists $\delta > 0$ and $p_0=p_0(C)>0$ such that if $p\in(0,p_0)$, then
$$\PB_{\lambda_c/p}\Big( \delta_{B \cap K^R_N} \big( \A_W,p \big) > N^{-\delta} \Big) \, \le \, N^{-C}$$
for every sufficiently large $N \in \N$.
\end{lemma}

In order to prove Lemma~\ref{reveal:A}, we first partition $R_{N+2}$ into $(N+2)^2$ squares of side-length~$1$ in the canonical way, and denote these squares by $S_1,\ldots,S_{(N+2)^2}$. Define $\AA_\ell$ to be the annulus centred at the origin, and consisting of all points with $\ell_\infty$-norm between $\ell$ and $2\ell$, and for each $1\le i \le (N+2)^2$, let $\AA_\ell(S_i)$ denote $\AA_\ell$ shifted to be concentric to $S_i$.
Let ${\C}\big( \AA_\ell(S_i),\eta \big)$ denote the (monotone decreasing) event that there is a loop of vacant space in $\AA_\ell(S_i)$; equivalently, it is the event that there is no path between the two faces of $\AA_\ell(S_i)$ using only points of~$D(\eta)$.

Now, consider the $t = \lfloor \log_4 (N / 4) \rfloor$ annuli $\AA_{\ell(1)}(S_i), \ldots, \AA_{\ell(t)}(S_i)$, where $\ell(j) = 4^j$. Note that the distance between $\AA_{\ell(j)}(S_i)$ and $\AA_{\ell(j+1)}(S_i)$ is at least 2 for each $j$, so the events $\C \big( \AA_{\ell(j)}(S_i),\eta)$ are independent.

The following lemma easily implies Lemma~\ref{reveal:A}.

\begin{lemma}\label{reveal:v}
For every $C > 0$, there exists $\delta > 0$ and $p_0=p_0(C)>0$ such that if $p\in(0,p_0)$, then for every $S_i$ that intersects $K_N^R$
\begin{equation}\label{propgives}
\PB_{\lambda_c/p}\Big( \delta_{B \cap S_i}\big( \A_W,p \big) > N^{-\delta} \Big) \; \le \; N^{- C}
\end{equation}
for every sufficiently large $N \in \N$.
\end{lemma}

\begin{proof}[Proof of Lemma \ref{reveal:v}]
First note that, for a given $B$ and $S_i \subset K^R_N$, if the event $\C\big( \AA_{\ell(j)}(S_i),\eta \big)$ occurs for some $1 \le j \le t$, then no point $v \in B\cap S_i$ will be queried by the algorithm $\A_W$. This follows because $v \in K^R_N$, so $v$ is $\ell_\infty$-distance at least $N/2$ from the left edge of $R_N$. Thus, by independence,
\begin{equation}\label{eqdelta}
\delta_{B \cap S_i}\big( \A_W,p \big) \,\le\, \prod_{j = 1}^t \bigg( 1 \,-\, \Pr^B_p\Big(  \C\big( \AA_{\ell(j)}(S_i),\eta \big) \Big) \bigg).
\end{equation}
Next, by Proposition~\ref{Bprop} and the FKG inequality (together with the union bound) it follows that for any $\gamma > 0$, if $p\in(0,p_0)$, then 
\begin{equation}\label{Xprob}
\PB_{\lambda_c/p} \bigg(  \Pr^B_p \Big( \C\big( \AA_\ell(S_i),\eta \big) \Big) \ge c'' \bigg) \, \ge \, 1 - 4\gamma,
\end{equation}
for all sufficiently large $\ell \in \N$, where $c'' = (c')^4$ and $p_0=p_0(\gamma) > 0$ are constants given by Proposition~\ref{Bprop}.

Using~\eqref{Xprob}, we have that
\begin{equation}\label{fewX}
\PB_{\lambda_c/p} \left( \Big| \Big\{ j \ge t/4 \,:\, \Pr^B_p\Big( \C\big( \AA_{\ell(j)}(S_i),\eta \big) \Big) \ge c'' \Big\} \Big| \,\le\, \frac{t}{2} \right) \, \le \, 2^t (4\gamma)^{t/4} \, \le \, N^{-C},
\end{equation}
if $\gamma = \gamma(C)$ is sufficiently small, and $N$ (hence also $t$) is sufficiently large, as required.

Finally, if $B$ is such that $\Pr^B_p \big( \C\big( \AA_{\ell(j)}(S_i),\eta \big) \big) \ge c''$ for at least $t/2$ of the annuli $\AA_{\ell(j)}(S_i)$, then for this $B$,
\begin{equation}\label{manyX}
\prod_{j = 1}^t \bigg( 1 \,-\, \Pr^B_p\Big(  \C\big( \AA_{\ell(j)}(S_i),\eta \big) \Big) \bigg) \,\le \, \big( 1 - c'' \big)^{t/2} \, \le \, N^{-\delta},
\end{equation}
for some (small) $\delta > 0$.

Combining~\eqref{eqdelta},~\eqref{fewX} and~\eqref{manyX}, it follows that
\[
\PB_{\lambda_c/p}\Big( \delta_{B \cap S_i}\big( \A_W,p \big) \, > \, N^{-\delta} \Big) \,\le\, N^{-C},
\]
as required.
\end{proof}

It is now easy to deduce Lemma~\ref{reveal:A}.

\begin{proof}[Proof of Lemma~\ref{reveal:A}]
Since there are exactly $(N+2)^2/2 \le N^2$ squares $S_i$ in $K^R_N$, it follows immediately from Lemma~\ref{reveal:v}, and the union bound, that 
\[
\PB_{\lambda_c/p}\Big( \delta_{B \cap K^R_N}\big( \A_W,p \big) > N^{-\delta} \Big) \; \le \; N^{- C + 2},
\]
as required.
\end{proof}

We can now deduce Theorem~\ref{BpNS}.

\begin{proof}[Proof of Theorem~\ref{BpNS}]
We prove that for each sufficiently small $p>0$, the model $\PP^B_p$ is NS$_p$ for $\PB_{\lambda_c/p}$-almost every $B$. To do so we must prove that the sequence of crossing functions $(f^B_N)_{N\ge 1}$ is NS$_p$ for $\PB_{\lambda_c/p}$-almost every $B$, for all sufficiently small $p > 0$.

According to Lemma~\ref{reveal:A} and Borel-Cantelli, there are $\delta > 0$ and $p_0$ such that for every $p\in(0,p_0)$ we have
\begin{equation}\label{eq:finite}
\PB_{\lambda_c/p}\Big(\delta_{B\cap K_N^R}\big( \A_W,p \big) > N^{-\delta}\text{ for at most finitely many }N\Big)=1.
\end{equation}
By symmetry, it follows that also $\delta_{B\cap K_N^L}\big( \A^*_W,p \big) \ge N^{-\delta}$ for at most finitely many values of $N$, with probability one, and hence
$$
\Big( \delta_{B\cap K_N^R}(\A_W,p) + \delta_{B\cap K_N^L}\big( \A^*_W,p \big) \Big) \big( \log N \big)^6 \; \to \; 0
$$
as $N \to \infty$, for $\PB_{\lambda_c/p}$-almost every $B$. Since $f^B_N$ is monotone, we may apply Theorem~\ref{algthm}, which implies that $(f^B_N)_{N \ge 1}$ is NS$_p$ with probability one, as required.
\end{proof}

\subsection{The Poisson Boolean model is noise sensitive}

We are finally ready to deduce Theorem~\ref{noise}; as we remarked in Section~\ref{sketchsec}, it follows easily from Theorem~\ref{BpNS} and Proposition~\ref{varBp}. The key observation is that the following two constructions are equivalent:
\begin{enumerate}
\item[$(a)$] Pick $\eta$ according to the measure $\PB_{\lambda_c}$, construct $\eta^\eps$ by deleting each element of $\eta$ independently with probability $\eps$, and add a new independent configuration picked according to the measure $\PB_{\eps \lambda_c}$. We consider the pair $(\eta,\eta^\eps)$.
\item[$(b)$] Pick $B$ according to the measure $\PB_{\lambda_c/p}$, and let $\eta$ be a $p$-subset of $B$.  Setting $\eps' = \eps/(1-p)$, construct $\eta^{\eps'}$ by re-randomizing the second step with probability $\eps'$, independently, for every $v \in B$. We consider the pair $(\eta,\eta^{\eps'})$.
\end{enumerate}
It is easy to see that picking $(\eta,\eta^\eps)$ according to the first construction is equivalent to picking $(\eta,\eta^{\eps'})$ according to the second construction.

Now, recall that $f^G_N$ is the function, defined on subsets $\eta$ of the plane $\RR^2$, which encodes whether or not there is a horizontal crossing of $R_N$ in the occupied space $D(\eta) \cap R_N$. To prove the noise sensitivity of the Poisson Boolean model we must prove that for every $\eps>0$
\begin{equation}\label{aim}
\lim_{N \to \infty}  \ExB_{\lambda_c} \big[ f^G_N(\eta) f^G_N ( \eta^{\eps}) \big] - \ExB_{\lambda_c} \big[ f^G_N(\eta) \big]^2=0,
\end{equation}
where the pair $(\eta,\eta^\eps)$ is obtained by the first construction above.

\begin{proof}[Proof of Theorem~\ref{noise}]
Fix $\eps>0$. We have, by the observations above,
$$
\ExB_{\lambda_c} \big[ f^G_N(\eta) f^G_N ( \eta^{\eps}) \big] - \ExB_{\lambda_c} \big[ f^G_N(\eta) \big]^2 \;  = \; \ExB_{\lambda_c/p} \Big[ \Ex^B_p\big[ f^B_N(\eta) f^B_N (\eta^{\eps'}) \big] \Big] \,-\,  \ExB_{\lambda_c/p} \Big[ \Ex^B_p\big[ f^B_N (\eta) \big] \Big]^2
$$
\begin{equation}\label{final}
\hspace{2.2cm} = \; \ExB_{\lambda_c/p}  \Big[ \Ex^B_p\big[ f^B_N (\eta) f^B_N (\eta^{\eps'}) \big] - \Ex^B_p\big[ f^B_N (\eta) \big]^2 \Big]  \,+\, \VarB_{\lambda_c/p} \Big( \Ex^B_p\big[ f^B_N (\eta) \big] \Big).
\end{equation}
Proposition~\ref{varBp} shows that the second term in~\eqref{final} can be made arbitrarily small by taking $p > 0$ sufficiently small and $N$ sufficiently large. Theorem~\ref{BpNS} implies in turn that for small (but fixed) $p>0$, the first term converges to $0$ as $N\to \infty$. Since~\eqref{final} holds for every $p > 0$, it follows that the limit~\eqref{aim} holds, and, since $\eps>0$ was arbitrary, the Poisson Boolean model is noise sensitive at criticality, as claimed.
\end{proof}

We observe that in the extremal case when $\eps=1-p$, \eqref{final} reduces to
$$
\ExB_{\lambda_c} \big[ f^G_N(\eta) f^G_N ( \eta^{1-p}) \big] - \ExB_{\lambda_c} \big[ f^G_N(\eta) \big]^2 \;  = \; \VarB_{\lambda_c/p} \Big( \Ex^B_p\big[ f^B_N (\eta) \big] \Big).
$$
Thus, as an immediate consequence of Theorem~\ref{noise} we obtain the following strengthening of Proposition~\ref{varBp}.

\begin{cor}\label{cor:noise}
For every $p\in(0,1)$
$$
\lim_{N \to \infty} \, \VarB_{\lambda_c/p} \Big( \P^B_p\big(H ( \eta, R_N, \bullet)\big) \Big)\, =\, 0.
$$
\end{cor}

Apart from its independent interest, we shall use Corollary~\ref{cor:noise} in the next section to obtain a quantitative bound on the noise sensitivity exponent associated with the Poisson Boolean model.

\section{Quantitative noise sensitivity}\label{sec:QNS}

Several years after the introduction of noise sensitivity in~\cite{BKS}, an important breakthrough was obtained by Schramm and Steif~\cite{SS}, who established a direct connection between revealment and the Fourier spectrum of a function via the use of randomized algorithms\footnote{In a randomized algorithm, the bit to be queried next is chosen according to a probability distribution that is allowed to depend on the information received so far.}. Recall Lemma~\ref{lem:NSequiv}, which shows that a sufficiently strong bound on the Fourier coefficients is enough to deduce noise sensitivity. We here outline how the `randomized algorithm' method of~\cite{SS} can be used to prove Corollary~\ref{cor:QNS}, and hence strengthen Theorem~\ref{noise}.

Let $p = 1/2$, and observe that, by Corollary~\ref{cor:noise}, the final term in the right-hand side of~\eqref{final} vanishes as $N \to \infty$, and does not depend on $\eps$. As a consequence, in order to prove Corollary~\ref{cor:QNS} it suffices to show that  there exists $\alpha > 0$ such that, with $\eps(N)=N^{-\alpha}$,
\begin{equation}\label{eq:QNSinPP}
\lim_{N\to\infty}\Ex^B_{1/2}\big[f_N^B(\eta)f_N^B(\eta^{\eps(N)})\big]-\Ex^B_{1/2}\big[f_N^B(\eta)\big]^2=0,\quad\text{for $\PB_{2\lambda_c}$-almost every $B$}.
\end{equation}
In other words, we need to prove that the noise sensitivity exponent for the $\PP_{1/2}^B$ model is (a.s.) bounded away from zero. In particular, this means that we do not need to extend the approach of~\cite{SS} from the uniform case to the biased setting; Theorem~\ref{noise} (via Corollary~\ref{cor:noise}) does the job for us.

The randomized algorithm method is based on the following relation between Fourier coefficients of a function and the revealment of an algorithm.

\begin{prop}[Schramm and Steif~\cite{SS}]\label{prop:SS}
Let $f:\{0,1\}^n\to\RR$ be a function and let $\A$ be a randomized algorithm determining $f$. For every $k\in[n]$ we have
$$
\sum_{|S|=k}\hat{f}(S)^2 \, \le \, k \cdot \| f \|^2_2 \cdot \delta_{[n]}(\A,1/2),
$$
where $\|f\|_2$ denotes the $L^2$-norm of $f$ with respect to uniform measure.
\end{prop}

We continue with a rough sketch of how Proposition~\ref{prop:SS} can be used to deduce~\eqref{eq:QNSinPP}, which in turn implies Corollary~\ref{cor:QNS}. We proceed as follows:
\begin{enumerate}[\qquad 1.]
\item First define a suitable algorithm $\A$ to be used; in fact, it is straightforward to adapt the algorithm which was used in~\cite{SS} to the continuum setting. Roughly speaking, the algorithm chooses a starting point uniformly at random from the middle third of the left-hand side of the square $R_N$, and then investigates (in two stages) whether or not there is an occupied crossing of $R_N$ originating above the starting point, and whether or not there is one originating below it.
\item Partition $R_{N+2}$ into $n = (N+2)^2$ unit squares $S_1, \ldots, S_n$, and observe that if the algorithm is to examine a point in $B\cap S_i$, then the exploration path from the starting point must reach $S_i$. The probability that this occurs can be bounded by the probability of a one-arm event originating from the square $S_i$. (Some care needs to be taken with squares within distance, say, $N^{1/2}$ of the boundary of $R_N$, but this is easily adjusted for.)
\item For small $p>0$, an estimate on the probability of the one-arm event was obtained in Lemma~\ref{reveal:v}, based on Proposition~\ref{Bprop}. To estimate $\delta_{[n]}(\A,1/2)$, we need to extend this statement to cover also $p=1/2$. Note that Corollary~\ref{cor:noise} together with Chebychev's inequality gives a statement very similar to Proposition~\ref{Bprop} (valid for all $p\in(0,1)$), but for crossings of the square $R_N$ instead of the rectangle $R_{N\times tN}$. Perhaps the easiest way to convince oneself that the statement holds also for rectangles, is to observe that there is nothing special about the choice of a square in our study of the sensitivity of crossings. Indeed, one may easily verify that the proof of Theorem~\ref{noise} goes through (practically word for word)  in the case where $R_N$ is replaced by $R_{N\times tN}$. Consequently, Corollary~\ref{cor:noise} also holds for rectangles, and hence for every $t > 0$ there exists a constant $c = c(t) > 0$ such that, for every $p \in (0,1)$,
$$
\lim_{N\to\infty}\PB_{\lambda_c/p}\Big(\Pr_p^B\big(H(\eta,R_{N\times tN},\bullet)\big)\not\in(c,1-c)\Big)=0.
$$
A bound on the one-arm event centred around a unit square $S_i$ is now easily obtained for $p=1/2$, exactly as in Lemma~\ref{reveal:v}.
\item Using the union bound and the Borel-Cantelli lemma, we conclude (exactly as in Section~\ref{ProofSubsec}) that there exists $\delta > 0$ such that, for all but finitely many $N$, the following holds for $\PB_{2\lambda_c}$-almost every $B$: given any unit square $S$ in $R_N$, the probability (in $\Pr_{1/2}^B$) of the one-arm event centred around $S$ is at most $N^{-\delta}$.
\item An upper bound on the revealment of the algorithm was in step~2 argued to be given by the (maximal) probability of the one-arm event around a unit square $S_i$. It follows that there exists $\delta>0$ such that
$$
\PB_{2\lambda_c}\Big(\delta_{B\cap R_{N+2}}(\A,1/2)>N^{-\delta}\text{ for at most finitely many }N\Big)=1.
$$
\item Via Proposition~\ref{prop:SS}, it follows that for each $\alpha<\delta/2$
\begin{equation}\label{eq:QFourier}
\sum_{0<|S|\le N^{\alpha}}\hat{f}^B_N(S)^2\;\le\;\sum_{0<k\le N^{\alpha}}k \cdot \| f^B_N \|^2_2 \cdot \delta_{[n]}(\A,1/2)\;\le\; N^{2\alpha}\cdot N^{-\delta} \;\to\;0,
\end{equation}
for $\PB_{2\lambda_c}$-almost every $B$.
\item As is easily verified (see the proof of Lemma~\ref{lem:NSequiv}), the noise sensitivity exponent for a sequence of Boolean functions $(f_n)_{n\ge1}$ is at least $\alpha$ if $\lim_{n\to\infty}\sum_{0<|S|\le n^{\alpha}}\hat{f}_n(S)^2=0$. Consequently,~\eqref{eq:QFourier} implies that the noise sensitivity exponent for the sequence $(f_N^B)_{N\ge1}$ is at least $\delta/2$, for $\PB_{2\lambda_c}$-almost every $B$. Thus,~\eqref{eq:QNSinPP} holds.
\end{enumerate}
\smallskip
By the observations above, this completes (the sketch of) the proof of Corollary~\ref{cor:QNS}.

\section{Open problems}\label{opensec}

In this paper we have laid out a fairly general approach to the problem of proving noise sensitivity in models of Continuum Percolation, and we expect that our method could be extended to prove similar results in more general settings. In this section we shall state a few of these open problems.

\subsection{More general Poisson Boolean models}

The simplest extension of Theorem~\ref{noise} would be to the Poisson Boolean model with (bounded) random radii. To obtain this model, let $R > 0$ and fix an (arbitrary) distribution $\mu_R$ on $(0,R)$. Now let $\eta \subset \RR^2$ be chosen according to a Poisson point process, and place a disc of radius $r(x)$ on each vertex $x \in \eta$, where $r(x)$ is chosen according to $\mu_R$, independently for each vertex. There are various ways to perturb a configuration in this model. One may leave the positions of the points unaffected, but re-randomize some of the radii, or add and remove a small proportion of the balls, much like in this paper. We have foremost the latter choice in mind. Indeed, for this model the only missing ingredient is an RSW Theorem for the occupied space.

\begin{conj}
For every $R > 0$ and $\mu_R$, the Poisson Boolean model with random radii chosen according to $\mu_R$ is noise sensitive at criticality.
\end{conj}

An alternative generalization would allow us to use an arbitrary shape $S$ instead of a disc. Given such an $S \subset \RR^2$, and a set $\eta$ chosen according to a Poisson point process, place a copy of $S$ on every point $x \in \eta$; that is, set $D(\eta) = \bigcup_{x \in \eta} x + S$.

\begin{conj}
If $S$ is bounded and simply connected, then the Poisson Boolean model for $S$ is noise sensitive at criticality.
\end{conj}

Of course, one could construct a much more general model, in which a random shape (of random size) is placed on each point in $\eta$. We suspect that such a model will also exhibit the same behaviour. The corresponding question for higher dimensions is likely to be much harder.

\begin{qu}
Is the Poisson Boolean model noise sensitive at criticality in $d$ dimensions?
\end{qu}

\subsection{Voronoi percolation}

Given a set $\eta \subset \RR^2$, the Voronoi tiling of $\eta$ (see~\cite{BR}, for example) is constructed by associating each point of $\RR^2$ with the point of $\eta$ closest to it. We call the set of points associated to $x \in \eta$ in this way the Voronoi cell of $x$. In Voronoi percolation we choose a random subset of the cells, by colouring each blue with probability $p$, and say that the model percolates if there exists an unbounded component of blue space. Bollob\'as and Riordan~\cite{BRvor} proved that if $\eta$ is chosen according to $\PB_{\lambda}$, then this model has critical probability $1/2$.

Given a Voronoi tiling $V$ of $\RR^2$, let $f_N^V \colon \{0,1\}^{V_N} \to \{0,1\}$ be the function which encodes crossings of $R_N$, where $V_N$ denotes the collection of cells which intersect $R_N$. Say that Voronoi percolation is noise sensitive at criticality if $(f_N^V)_{N \ge 1}$ is NS, almost surely (in $\PB_{\lambda}$).

Benjamini, Kalai and Schramm~\cite[Section~5]{BKS} asked whether knowing the Voronoi tiling, but not the colouring, gives (a.s.) any information as to whether or not there exists a blue horizontal crossing of $R_N$. We make the following, complementary conjecture.

\begin{conj}
Voronoi percolation is noise sensitive at criticality.
\end{conj}

Alternatively, one could define noise sensitivity by resampling the Poisson point process, as well as the colouring; we expect the corresponding result to hold for this definition also.

\subsection{Stronger results for the Gilbert model}

In Corollary~\ref{cor:QNS} we show that it is possible to obtain quantitative estimates on how sensitive percolation crossings are to noise. Recently, \emph{very} strong results have been obtained by Schramm and Steif~\cite{SS} and Garban, Pete and Schramm~\cite{GPS} in the discrete case. In~\cite{SS} it was proven that the noise sensitivity exponents for bond percolation on the square lattice, and for site percolation on the triangular lattice, are positive. The latter was improved in~\cite{GPS} to show that the noise sensitivity exponent for the triangular lattice equals $3/4$. Such a precise result was possible to obtain due to the very precise information available on the decay rate of arms events. It would be interesting to determine the precise value of the noise sensitivity exponent in the continuum setting.



\begin{prob}\label{FS}
Determine for which $\alpha>0$ we with $\eps(N)=N^{-\alpha}$ have
\[
\lim_{N \to \infty} \ExB_{\lambda_c} \big[ f^G_N(\eta) f^G_{N}( \eta^{\eps(N)}) \big] - \ExB_{\lambda_c} \big[ f^G_N(\eta) \big]^2 \, =\, 0.
\]
\end{prob}

One possible application of a solution to Problem~\ref{FS} would be to Dynamical Continuum Percolation. To define this model, consider a set $\eta$ chosen according to a Poisson point process of intensity $\lambda_c$ in three dimensions (two space and one time), and suppose points of $\eta$ disappear at rate one. By Corollary~\ref{Cor:Alex}, at any given time there is (almost surely) no infinite component in $D(\eta)$; we therefore say that $t$ is an \emph{exceptional time} if there is an infinite component in $D(\eta)$ at time~$t$.

The following conjecture was proved for site percolation on the triangular lattice in~\cite{SS}, and for bond percolation on $\ZZ^2$ in~\cite{GPS}.

\begin{conj}
There exist exceptional times in Dynamical Continuum Percolation at criticality, almost surely.
\end{conj}

A related problem was studied by Benjamini and Schramm~\cite{BS2}. Alternatively, one might first choose the points of $\eta$ according to a Poisson point process of intensity $\lambda_c$, and then allow each to perform an independent Brownian motion (see also~\cite{BGS}); an interesting first step would be to prove a result corresponding to Theorem~\ref{noise} for this model.

\section*{Acknowledgements}

This research began during the Clay Mathematics Institute summer school on ``Probability and Statistical Physics in Two and more Dimensions" which took place in Buzios, Brazil in July, 2010. We would like to thank the organisers of that school, and the Clay Institute, for providing such a stimulating environment in which to work. We would also like to thank Yuval Peres for a useful conversation, and for encouraging us to work on the problem. The work was continued during a visit by D.~Ahlberg to IMPA, whom he thanks for their support and hospitality on several occasions. Finally, we would like to thank the two anonymous referees for their very careful and thorough reading of the paper.


\begin{thebibliography}{99}

\bibitem{Aha} R.~Aharoni, A problem in rearrangements of (0,1) matrices, \emph{Discrete Math.}, \textbf{30} (1980), 191--201.

\bibitem{Daniel} D.~Ahlberg, Partially observed Boolean sequences and noise sensitivity, submitted.

\bibitem{AK} R.~Ahlswede and G.O.H.~Katona, Graphs with maximal number of adjacent pairs of edges, \emph{Acta Math. Acad. Sci. Hungar.}, \textbf{32} (1978), 97--120.

\bibitem{Alex} K.S.~Alexander, The RSW theorem for Continuum Percolation and the CLT for Euclidean minimal spanning trees, \emph{Ann. Appl. Probab.}, \textbf{6} (1996), 466--494.

\bibitem{AS} N.~Alon and J.~Spencer, The Probabilistic Method (3rd edition), Wiley Interscience, 2008.

\bibitem{BBW} P.~Balister, B.~Bollob\'as and M.~Walters, Continuum Percolation with steps in the square or the disc, \emph{Random Structures Algorithms}, \textbf{26} (2005), 392--403.

\bibitem{BS1} I.~Benjamini and O.~Schramm, Conformal Invariance of Voronoi Percolation, \emph{Comm. Math. Phys.}, \textbf{197} (1996), 75--107.

\bibitem{BS2} I.~Benjamini and O.~Schramm, Exceptional planes of percolation, \emph{Prob. Theory Rel. Fields}, \textbf{111} (1998), 551--564.

\bibitem{BKS} I.~Benjamini, G.~Kalai and O.~Schramm, Noise sensitivity of Boolean functions and applications to percolation, \emph{Inst. Hautes Etudes Sci. Publ. Math.}, \textbf{90} (1999), 5--43.

\bibitem{BSW} I.~Benjamini, O.~Schramm and D.B.~Wilson, Balanced Boolean functions that can be evaluated so that every input bit is unlikely to be read. In \emph{STOC'05: Proceedings of the 37th Annual ACM Symposium on Theory of Computing}, 244--250. ACM, New York, 2005.

\bibitem{Bey} C.~Bey, An upper bound on the sum of squares of degrees in a hypergraph, \emph{Discrete  Math.}, \textbf{269} (2003), 259--263.

\bibitem{MGT} B.~Bollob\'as, Modern Graph Theory (2nd edition), Springer, 2002.

\bibitem{BRvor} B.~Bollob\'as and O.~Riordan, The critical probability for random Voronoi percolation in the plane is 1/2, \emph{Prob. Theory Rel. Fields}, \textbf{136} (2006), 417--468.

\bibitem{BR} B.~Bollob\'as and O.~Riordan, Percolation, Cambridge University Press, 2006.

\bibitem{BKKKL} J.~Bourgain, J.~Kahn, G.~Kalai, Y.~Katznelson and N.~Linial, The influence of variables in product spaces, \emph{Israel J. Math.}, \textbf{77} (1992), 55--64.

\bibitem{BGS} E.I.~Broman, C.~Garban and J.E.~Steif, Exclusion Sensitivity of Boolean Functions, \emph{Prob. Theory Rel. Fields}, \textbf{155} (2013), 621--663.

\bibitem{dC} D.~de Caen, An upper bound on the sum of squares of degrees in a graph, \emph{Discrete Math.}, \textbf{185} (1998), 245--248.

\bibitem{Cher} H.~Chernoff, A Measure of Asymptotic Efficiency for Tests of a Hypothesis Based on the sum of Observations, \emph{Ann. Math. Statistics}, \textbf{23} (1952), 493--507.

\bibitem{Fried} E.~Friedgut. Influences in product spaces: KKL and BKKKL revisited, \emph{Combin. Probab. Comput.}, \textbf{13} (2004), 17--29.

\bibitem{Gar} C.~Garban, Oded Schramm's contributions to noise sensitivity, \emph{Ann. Probab.}, \textbf{39} (2011), 1702--1767.

\bibitem{GPS} C.~Garban, G.~Pete, and O.~Schramm, The Fourier spectrum of critical percolation, \emph{Acta Math.}, \textbf{205} (2010), 19--104.

\bibitem{Gilbert} E.N.~Gilbert, Random Plane Networks, \emph{J. Soc. Indust. and Appl. Math.}, \textbf{9} (1961), 533--543.

\bibitem{Grim} G.~Grimmett, Percolation, 2nd edition, Springer-Verlag, Berlin, 1999.

\bibitem{HPS1} O.~H\"aggstr\"om, Y.~Peres and J.E.~Steif, Dynamical percolation, \emph{Ann. Inst. H. Poincare Probab. Statist.}, \textbf{33} (1997), 497--528.


\bibitem{HPS2} A.~Hammond, G.~Pete and O.~Schramm, Local time on the exceptional set of dynamical percolation, and the Incipient Infinite Cluster, submitted, arXiv:1208.3826.


\bibitem{KKL} J.~Kahn, G.~Kalai and N.~Linial, The influence of variables on Boolean functions, 29th Annual Symposium on Foundations of Computer Science, (68--80), 1988.


\bibitem{Kell} N.~Keller, A simple reduction from a biased measure on the discrete cube to the uniform measure, \emph{Europ. J. Combinatorics}, \textbf{33} (2012), 1943--1957.

\bibitem{KK} N.~Keller and G.~Kindler, Quantitative Relation Between Noise Sensitivity and Influences, \emph{Combinatorica}, to appear.

\bibitem{KMS} N.~Keller, E.~Mossel and T.~Schlank, A Note on the Entropy/Influence Conjecture, \emph{Discrete Math.}, \textbf{312} (2012), 3364--3372.



\bibitem{MR} R.~Meester and R.~Roy, Continuum Percolation, Cambridge University Press, 1996.

\bibitem{MS} M.V.~Menshikov and A.F.~Sidorenko, Coincidence of Critical Points in Poisson Percolation Models, \emph{Teor. Veroyatnost. i Primenen.}, \textbf{32} (1987), 603--606.


\bibitem{Paley} R.E.A.C.~Paley, A remarkable series of orthogonal functions, \emph{Proc. London Math. Soc.}, \textbf{34} (1932), 241.


\bibitem{Roy} R.~Roy, The Russo Seymour Welsh theorem and the equality of critical densities and the dual critical densities for Continuum Percolation, \emph{Ann. Probab.}, \textbf{18} (1990), 1563--1575.


\bibitem{SS} O.~Schramm and J.~Steif, Quantitative noise sensitivity and exceptional times for percolation, \emph{Ann. Math.}, \textbf{171} (2010), 619--672.

\bibitem{Steif} J.~Steif, A survey of dynamical percolation, \emph{Fractal geometry and stochastics}, IV, Birkhauser, 145--174, 2009.


\bibitem{Walsh} J.L.~Walsh, A closed set of orthogonal functions, \emph{Amer. J. Math.}, \textbf{55} (1923), 5 pp.


\end{thebibliography}
\end{document}